%% file: main.tex
\documentclass{article}
\usepackage{graphicx} 

\title{Thresholds for Reconstruction of Random Hypergraphs \\ From Graph Projections}

\author{Guy Bresler\thanks{Supported in part by NSF Career award CCF-1940205.}
\and Chenghao Guo\thanks{Supported in part by NSF TRIPODS grant DMS-2022448, NSF Career award CCF-1940205, CCF-2131115 and the MIT-IBM Watson AI Lab.} \and Yury Polyanskiy\thanks{Supported in part by the NSF grant CCF-2131115 and the MIT-IBM Watson AI Lab.}}
\date{}

\input{macros}

\begin{document}
\maketitle

\begin{abstract}
The \emph{graph projection} of a hypergraph is a simple graph with the same vertex set and with an edge between each pair of vertices that appear in a hyperedge.  
We consider the problem of reconstructing a random $d$-uniform hypergraph from its projection. Feasibility of this task depends on $d$ and the density of hyperedges in the random hypergraph. For $d=3$ we precisely determine the threshold, while for $d\geq 4$ we give bounds. All of our feasibility results are obtained by exhibiting an efficient algorithm for reconstructing the original hypergraph, while infeasibility is information-theoretic. 

Our results also apply to mildly inhomogeneous random hypergrahps, including hypergraph stochastic block models (HSBM). A consequence of our results is an optimal HSBM recovery algorithm, improving on \cite{gaudio2023community}. 
\end{abstract}

\input{Introduction}


\input{MainIdeas}


\input{Upperbound}

\input{growth}

\input{FutureDirections}

\bibliographystyle{alpha}
\bibliography{ref}

\appendix

\input{Lowerbound}

\input{ProofNaiveAlgo}

\input{deferred}

\end{document}

%% file: macros.tex
\usepackage[utf8]{inputenc}
\usepackage{amsmath}
\usepackage{amssymb,amsfonts,amsthm}
\usepackage{mathrsfs}
\usepackage{algorithm}
\usepackage{xcolor}
\usepackage{tikz}
\usepackage{fullpage}
\usepackage{dsfont}
\usepackage{thmtools}
\usepackage{thm-restate}
\usepackage{enumitem}
\usepackage[noend]{algpseudocode}
\usepackage{subcaption}
\usepackage{graphicx}
\usetikzlibrary{shapes,arrows,backgrounds,fit,positioning,calc}
\usepackage{subcaption}

\usepackage{hyperref}
\usepackage{cleveref}
\usepackage{xcolor}
\hypersetup{
  colorlinks   = true, 
  urlcolor     = blue, 
  linkcolor    = blue, 
  citecolor   = blue 
}

\usepackage{tikz}
\usetikzlibrary{decorations.pathreplacing}
\usetikzlibrary{arrows,positioning}
\DeclareRobustCommand\shape{
 \lower5pt\hbox{
 \hskip-7pt
  \tikzset{circ/.style={circle, draw, fill=black, scale=.15}}
  \begin{tikzpicture}[semithick,scale=.3]
  \node (l1) at (0,.5) [circ]{};
  \node (l3) at (0.5,0.3) [circ]{};
  \draw[-] (l1) to node [auto] {} (l3);
    \end{tikzpicture}
  \hskip-8pt}
}
\newtheorem{theorem}{Theorem}[section]

\newtheorem{lemma}{Lemma}[section]
\newtheorem{corollary}[lemma]{Corollary}
\newtheorem{remark}{Remark}[]
\newtheorem{definition}{Definition}

\newcommand{\var}{\mathsf{Var}}
\DeclareMathOperator*{\E}{{\rm I}\kern-0.18em{\rm E}}
\renewcommand{\Pr}{\,{\rm I}\kern-0.18em{\rm P}}
\newcommand{\1}{\mathds{1}}
\newcommand{\ind}[1]{\1\{#1\}}

\newcommand{\bino}{\mathsf{Binomial}}

\newcommand{\defeq}{\triangleq}

\newcommand{\inv}{^{-1}}

\newcommand{\B}{\Big}
\renewcommand{\b}{\big}

\newcommand{\ER}{{Erd\H os-R\'enyi}}

\newcommand{\cA}{\mathcal{A}}

\newcommand{\cG}{\mathcal{G}}

\newcommand{\cE}{\mathcal{E}}
\newcommand{\cH}{\mathcal{H}}

\newcommand{\cov}{\mathsf{Cov}}

\newcommand{\cS}{\mathcal{S}}

\newcommand{\hG}{H}
\newcommand{\shG}{K}
\newcommand{\rhG}{\cH}
\newcommand{\hE}{\cE}
\newcommand{\he}{h}
\newcommand{\pG}{{G_p}}
\newcommand{\pE}{E_p}
\newcommand{\cliG}{{\rhG_c}}
\newcommand{\cliE}{\cE_c}
\newcommand{\proj}{\mathrm{Proj}}

\newcommand{\prim}{{\mathrm{Proj}^{-1}}}
\newcommand{\cli}{\mathrm{Cli}}
\newcommand{\nei}[1]{N_{#1}}
\newcommand{\aG}{G_{a,d}}
\newcommand{\cdel}[1]{\delta^*_{#1}}
\newcommand{\cliA}{\cA_c}
\newcommand{\naH}{{\shG_f}}
\newcommand{\aut}{\mathrm{aut}}

\newcommand{\bG}{\shG_b}

\newcommand{\hsbm}{\mathrm{HSBM}}
\newcommand{\grow}{\mathrm{Grow}}

\DeclareMathOperator*{\argmin}{arg\,min}

\newcommand\numberthis{\addtocounter{equation}{1}\tag{\theequation}}
\newcommand{\cth}{2-connectivity threshold}
\newcommand{\tcth}{\frac{d-1}{d+1}}
\newcommand{\ath}{\delta_d^a}

%% file: Introduction.tex
\section{Introduction}

Graphs and hypergraphs are fundamental structures in diverse fields such as computer science, mathematics, social science, and biology, supporting a wide range of theoretical and applied research areas. 
Hypergraphs generalize graphs, with hyperedges consisting of subsets of the vertices.
Because interactions between entities often occur in groups, such as people dining together or items added to an online shopping cart, many phenomena are best captured using hypergraphs. At the same time, the vast majority of graph algorithms are designed for simple graphs, where edges constitute a \emph{pairwise} relationship.

Given a hypergraph $H$, one can construct a graph $G$ by including an edge between each pair of vertices that appear in some hyperedge in $H$. This corresponds to placing in $G$ a clique on the vertices appearing in each hyperedge of $H$. We say that $G$ is the \emph{projection} of $H$. 

Projecting hypergraphs onto graphs and leveraging graph algorithms is a common strategy for solving problems on hypergraphs. 
This approach has been pursued especially in the domain of community detection within the hypergraph stochastic block model, where algorithms aim to reconstruct communities from similarity matrices, a form of pairwise hypergraph projection
\cite{kim2018stochastic,cole2020exact,gaudio2023community}.
Similar methodologies also exist in hypergraph matching, where the optimal soft matching can be obtained by considering pairwise interactions \cite{zass2008probabilistic}. 
More generally in graph data processing, projection of hypergraphs is used to improve storage efficiency and interpretability, or simply to allow use of existing data structures and algorithms.

There can be many different hypergraphs that project to a given graph $G$, and thus the projection operation is often \emph{lossy}. 
It is not at all clear when projecting a hypergraph and solving some problem on the projected graph is optimal, and in general this depends both on the task and on the hypergraph. One scenario in which projecting to a simple graph does not degrade performance, \emph{whatever the task}, is if it is possible to efficiently reconstruct the hypergraph from the projected graph. 
This motivates the following basic question: under what conditions does projecting a hypergraph result in information loss, and conversely, when can a hypergraph be recovered from its projection?



Beyond serving as a justifying principle for employing hypergraph-to-graph projections in algorithm creation, the task of recovering hypergraphs from their graph projections arises naturally in network analysis. 
The phenomenon of intrinsic hypergraphs appearing as projected graphs is common in real-world networks \cite{zhou2006learning,latapy2008basic,williamson2020random,battiston2020networks}. For instance, two scientists are listed as co-authors on Google Scholar because they collaborate on the same paper \cite{newman2004coauthorship}, two people send emails to each other because they are working on the same project \cite{klimt2004introducing}. In such scenarios, direct methods for detecting higher order interactions are often unavailable, which highlights the importance of hypergraph recovery.

Prior research on this problem has been focused on designing algorithms with good empirical performance; none of the following works have theoretical guarantees. In \cite{young2021hypergraph,lizotte2023hypergraph}, the authors assumed a prior distribution over hypergraphs, and try to sample from the posterior to approximate the original hypergraph. 
The work \cite{wang2022supervised} aims to recover a hypergraph from its graph projection, for a general distribution over initial hypergraph given access to another hypergraph independently sampled from the same distribution. A scoring method was then used to select hyperedges based on their similarity to the sampled hypergraph. 


In this work we aim to provide a deeper understanding of the conditions under which hypergraphs can be recovered from graph projections. 
We study the problem of recovering a \emph{random} $d$-uniform hypergraph, where all hyperedges are of size $d$, from its graph projection. For $d=3$ we determine a precise threshold in the hyperedge density at which recovery is feasible, and give an efficient algorithm to do so when it is. For $d\geq 4$ we provide bounds on the hyperedge density. 
Our analysis relies on analyzing the local structure of random hypergraphs, and in the process we identify useful structural properties of random hypergraphs.

Our results hold also for mildly inhomogeneous random hypergraphs, where edge probabilities may be non-uniform but are all within constant factors of one another. This includes the hypergraph stochastic block model (HSBM).
The question of determining the information-theoretic recovery thresholds for HSBM, given the similarity matrix, was previously posed as an open problem in \cite{gaudio2023community}.
As a by-product of our results, we solve the open problem showing that the information theoretic threshold of HSBM, given the similarity matrix, coincides with that of HSBM given the original hypergraph.\footnote{One of the original motivations of the present paper was to disprove
the claim that the two thresholds are different, made in \cite{gaudio2023community-v1}. Later versions \cite{gaudio2023community} replace this with the statement that the threshold for HSBM recovery from the similarity matrix is open.} This is proven by a reduction that recovers the original hypergraph given the similarity matrix.

\subsection{Hypergraph Reconstruction Problem Formulation}
Before describing our problem formulation we require a couple of definitions.

\subsubsection{Random hypergraphs}
We define the following model of random hypergraphs, generalizing the \ER\ random graph.\footnote{This definition of random hypergraph is equivalent to the definition of random $d$-complex \cite{toth2017handbook} except here we use the language of hypergraphs instead of simplicial complexes. The model considered in \cite{young2021hypergraph} is an inhomogeneous generalization of our model where each hyperedge has a distinct probability of appearing. Projection of random hypergraphs was also proposed as a way to simulate network data \cite{williamson2020random}. 
}

A \emph{random $d$-hypergraph} $\rhG(n,d,p)=([n],\hE_{\hG})$ is a $d$-uniform hypergraph where every size-$d$ hyperedge in $\binom{[n]}{d}$ is included in $\hE_{\rhG}$ with probability $p$ independently. We will use the parameterization $$p=n^{-d+1+\delta}\,,$$ so that the expected degree of each node is on the order $n^\delta$.\footnote{Constant factors do not affect any result of the paper. All of our results also holds with possibly different probabilities of inclusion at different edges, as long as the probability is $\Theta(n^{-d+1+\delta})$. } 

\subsubsection{Graph projection}
Given a hypergraph $H=([n],\hE)$, we consider the \emph{projection} $\proj(H)$ which takes
$d$-uniform hyperedges to ordinary (pairwise, undirected) edges by simply including an edge if  both its endpoints are in a hyperedge:
\[
\proj(\hE) \defeq \b\{(i,j)\in \textstyle{\binom{ [n]}{2}}: i,j\in \he\text{ for some } \he\in \hE\b\}.
\]
Here we overload notation, using $\proj$ both for projection of a set of hyperedges and for the projected graph. 
A random hypergraph $\rhG$ results in a \emph{random projected graph} $\pG=\proj(\rhG)=([n], \pE=\proj(\hE_{\rhG}))$. 
%
For one hyperedge $\he$, we use $\proj(\he)$ to denote $ \proj(\{\he\})$. For a simple graph $G$, we say a hypergraph in $\prim (G)$ is a \emph{preimage} or a \emph{clique cover} of $G$. We will frequently use the fact that projection commutes with union: $\proj(C_1\cup C_2) = \proj(C_1)\cup \proj( C_2)$. 

Our goal is to recover the original hypergraph $\rhG$ from the projected graph $\pG$.

\subsubsection{Exact Recovery}
We say that an algorithm $\cA:\{0,1\}^{[n]\choose 2}\to \{0,1\}^{[n]\choose d}$ mapping a projected graph $\pG$ to a $d$-uniform hypergraph can achieve (asymptotically) \emph{exact recovery} if 
\begin{equation}\label{e:recovery}
    \Pr\b(\cA(\proj(\rhG)) = \rhG\b)=1-o_n(1)\,.
\end{equation}

\begin{remark}\label{rem:delta-range}
   We parameterize $p=p(\delta,d,n)=n^{-d+1+\delta}$ so that the expected degree of a node is $\Theta(n^\delta)$.
    The problem of exact recovery is only interesting when $0\le \delta\le 1$. When $\delta<0$, with high probability $\pG$ only consists of isolated $d$-cliques, so exact recovery is trivial. When $\delta>1$, with high probability $\pG$ is complete, so exact recovery is impossible.
\end{remark}


\paragraph{Information-theoretic Versus Algorithmic Feasibility.}
    The existence of an algorithm $\cA$ satisfying
    \eqref{e:recovery}
    answers the question of whether the projection operator loses information.
Exact recovery is \emph{information theoretically possible} for a certain $\delta$ if there exists an algorithm $\cA$ that can do exact recovery regardless of time complexity. 
When exact recovery is information theoretically possible, we wish to find an efficient algorithm. Exact recovery is said to be \emph{efficiently achievable} for a certain $\delta$ if there exists a polynomial-time algorithm $\cA$ that can achieve exact recovery.

\subsection{Results}

Before describing our results, it will be helpful to gain a qualitative understanding of how the density $p=n^{-d+1+\delta}$ impacts the difficulty of exact recovery. The main intuition is that as we make the hypergraph denser, recovery from the projected graph gets more difficult as projections of different hyperedges begin to overlap. The extreme case where the projected graph is complete was mentioned in Remark~\ref{rem:delta-range}. This intuition is formalized by the following lemma, which is proved in Appendix~\ref{sec:monotone}.

\begin{lemma}[Monotonicity in $\delta$]\label{lem:monotone}
    For any $d\ge 4$ and any $0\le \delta_1<\delta_2\le 1$, if exact recovery is information theoretically possible (or efficiently achievable) when $\delta=\delta_2$, then exact recovery is also information theoretically possible (or efficiently achievable) for $\delta=\delta_1$.  
\end{lemma}
The lemma is proved via a simple reduction: given $G = \proj(\rhG)$ where $\rhG$ has density $p(\delta_1,n,d)$, we independently sample a random hypergraph $\rhG'$ so that $\rhG\cup \rhG'$ has density $p(\delta_2,n,d)$ and give algorithm $\cA$ (presumed to achieve exact recovery at $\delta_1$) the graph $\proj(\rhG)\cup \proj(\rhG') = \proj(\rhG\cup \rhG')$. We then remove the hyperedges in $\rhG'$ from the output of $\cA$, and this succeeds as long as $\rhG$ and $\rhG'$ have no hyperedges in common. This latter property holds for $d\geq 4$, but not for $d=3$.

It follows that for $d\ge 4$ there must exist a threshold $\cdel d$ above which exact recovery is possible and below which exact recovery is impossible. Formally, let
\[
\cdel d \defeq \inf \{\delta: \text{exact recovery is impossible at }\delta\}\,.
\]
We have the following corollary from the lemma above.
\begin{corollary}[Threshold for Exact Recovery]\label{cor:monotone}
    For $d\ge 4$, exact recovery is information theoretically possible for any $\delta<\cdel d$ and impossible for any $\delta>\cdel d$.
\end{corollary}

The statement of the corollary is also true for $d=3$, but this requires a different argument.
We determine the location of the threshold when $d=3$ and we also prove that exact recovery precisely at the threshold is impossible.

\begin{theorem}\label{thm:main-three}
    For $d=3$, there is an efficient algorithm for exact recovery when $\delta<2/5$ and exact recovery is information theoretically impossible when $\delta\ge 2/5$.
\end{theorem}

 For $d\ge 4$, as stated in the following two theorems, we demonstrate that the threshold $\cdel d$ must lie in a certain interval. Furthermore, we find an efficient algorithm (in fact, attaining the optimal probability of reconstruction error) in the regime where we show that exact recovery is possible. It is worth noting that our algorithm does not need to know $p$ as an input parameter.
The results are summarized in Table~\ref{table:delta_bounds}.

\begin{theorem}\label{thm:main-four-five}
    For $d=4,5$, there is an efficient algorithm for exact recovery when $\delta<1/2$ and exact recovery is information theoretically impossible when $\delta\ge \frac{2d-4}{2d-1}$.
\end{theorem}

\begin{theorem}\label{thm:main-large-d}
    For $d\ge 6$, there is an efficient algorithm for exact recovery when $\delta<\frac{d-3}{d}$ and exact recovery is information theoretically impossible when $\delta\ge \frac{d^2-d-2}{d^2-d+2}$.
\end{theorem}

\begin{table}[t]
\centering
\begin{tabular}{|c|c|c|}
\hline
Value of \(d\) & Lower Bound for \(\cdel d\) & Upper Bound for \(\cdel d\) \\ \hline
\(3\) & \(2/5\) & \(2/5\) \\
\(4\) & \(1/2\) & \(4/7\) \\
\(5\) & \(1/2\) & \(2/3\) \\
\(d\ge 6\) & \(\frac{d-3}{d}\) & \(\frac{d^2-d-2}{d^2-d+2}\) \\ \hline
\end{tabular}
\caption{Bounds for hyperedge density threshold \(\cdel d\).}
\label{table:delta_bounds}
\end{table}

For $d=4$ and $5$, we conjecture that the correct threshold is at $\frac{2d-4}{2d-1}$ (note this is the case for $d=3$). Our methodology enables proving the conjecture by verifying certain combinatorial properties for finitely many graphs, a check that can be carried out with computer assistance. However, the computation required is significant and we were unable to complete the computer verification. We elaborate on this in Section~\ref{s:mainIdeas}.

\input{HSBM}

\subsection{Notation}
We always use $H$ for hypergraphs, $\he$ for hyperedges and $\hE$ for a set of hyperedges. $G$ stands for a simple graph, $e$ is used to denote an edge, and $E$ denotes a set of edges. The \emph{size} of a graph (hypergraph) means the number of edges (hyperedges) in the graph (hypergraph). We often identify a graph or hypergraph simply by its edge set, which causes no ambiguity in the case that every vertex is in some edge (i.e., there are no isolated vertices).

All the probabilities $\Pr$ are in the probability space defined by the random $d$-hypergraph $\rhG(n,d,p)$. We denote by $X_{\hG}$ the random variable equal to the number of appearances of $\hG$ as a sub-hypergraph of $\rhG$.


%% file: HSBM.tex
\subsubsection{Application to Hypergraph Stochastic Block Model}
\label{s:HSBM}
We now discuss the application of our results to the Hypergraph Stochastic Block Model (HSBM).
As we explain momentarily, a byproduct of our result is that community detection from the graph projection of the HSBM is equivalent to community detection given the original HSBM hypergraph, and this is also equivalent to doing so given the similarity matrix (defined below). 

The model $\hsbm(d,n,q_1,q_2)$ describes a random $d$-uniform hypergraph on $n$ vertices, parameterized by $q_1$ and $q_2$. A sample $\rhG$
is generated as follows. First an assignment of labels $\sigma\in\{\pm 1\}^n$ for the vertices is sampled uniformly at random from all assignments with equal number of $+1$ and $-1$ ($n$ is assumed to be even). Conditional on $\sigma$, for each $\he\in\binom{[n]}{d}$, the hyperedge $\he=\{i_1,\cdots,i_d\}$ is included in $\rhG$ independently with probability
\[
\Pr(\he\in\rhG )=
\begin{cases}
    q_1 &\text{ if } \sigma_{i_1} = \sigma_{i_2} = \cdots = \sigma_{i_d}\\
    q_2 &\text{otherwise}\,.
\end{cases}
\]
The probabilities $q_1$ and $q_2$ are parameterized as 
$q_1 = \alpha\log n/\binom{n-1}{d-1}$ and $q_2 = \beta\log n/\binom{n-1}{d-1}$. 

In the \emph{community recovery} problem, we are given a sample hypergraph $\rhG\sim \hsbm(d,n,q_1,q_2)$ and we want to recover the assignment for all vertices (up to global sign flip).

The \emph{similarity matrix} $W$ of a hypergraph $H=([n],\hE)$  is defined to be 
\[
W_{ij} = |\{\he\in \hE: i,j\in \he\}|\,.
\]
In \cite{kim2018stochastic,cole2020exact,gaudio2023community}, the similarity matrix of the hypergraph is used as the algorithm input. A basic question is: does using the similarity matrix lose performance as compared to using the original hypergraph? Our result shows that this is not the case. Specifically, if there is an algorithm that recovers the assignment for some $d,\alpha$ and $\beta$ with the hypergraph as input, then there exists an algorithm that recovers the assignment for the same $d,\alpha$ and $\beta$ with the similarity matrix as input. 
This yields an algorithm for exact recovery given the similarity matrix that outperforms those in prior work.


\begin{theorem}
    For any $d, \beta$ and $\alpha$,
    given the similarity matrix $W$ of $\hsbm(d,n,q_1,q_2)$ where $q_1 = \alpha\log n/\binom{n-1}{d-1}$ and $q_2 = \beta\log n/\binom{n-1}{d-1}$, we can exactly recover the hypergraph with high probability.
\end{theorem}
\begin{proof}
In the HSBM parameter regime, the edge density is $\Theta(n^{-d+1}\log n)$, which is far below the critical threshold $n^{-d+1+\cdel d}$ and indeed also far below our \emph{lower bound} on the critical threshold (i.e., our algorithms succeed in this range). Note that the HSBM may not appear to be within the setting of this paper because: 
\begin{enumerate}
    \item The probability of having a hyperedge depends on the assignment of the nodes.
    \item There is a constant that differs across hyperedges, as well as a $\log n$ factor, in front of the probability.
\end{enumerate}
However, all of our achievability results below the critical threshold $p=n^{-d+1+\cdel d}$ only require an upper bound on the hyperedge probabilities, regardless of whether the probability depends on specific edges. For instance, in the proof of Lemma~\ref{lem:exp-dec}, we only used the fact that $p=O_n(n^{-d+1+\delta})$. In the regime of HSBM both $q_1$ and $q_2$ are $O_n(n^{-d+1+\delta})$, so the argument still holds. 
In this paper we nevertheless use the parameterization $p=n^{-d+1+\delta}$ for clarity of exposition.
\end{proof}


%% file: MainIdeas.tex
\section{Main Ideas}
\label{s:mainIdeas}
As a warm up and to introduce some notation and ideas, we first describe a simple algorithm that produces a hypergraph from a graph by including every possible hyperedge. This can result in a hypergraph that has many more hyperedges than the maximum \emph{a posteriori} (MAP) hypergraph, and therefore has far lower posterior probability. Correspondingly, this simple algorithm succeeds in a smaller range of edge densities than the MAP rule, however, this algorithm does turn out to succeed in a nontrivial parameter range. We then describe our algorithm for constructing the MAP hypergraph and the associated guarantees.

\input{NaiveAlgorithms}

\subsection{Information-Theoretically Optimal Algorithm: MAP}


Although fully determining the landscape of exact recovery is non-trivial, the optimal algorithm for the task is in fact not hard to describe. 
Given the projected graph $\pG = \proj(\rhG)$, the error probability upon outputting $\cA(\pG)$ is simply the complement of the probability that our guess was the true hypergraph,
\[
1-p_{\rhG|\pG}(\cA(\pG)|\pG) \,.
\]
Here $p_{\rhG|\pG}$ is the conditional probability mass function of the random hypergraph given the projected graph.
Therefore, if we do not worry about time complexity, the information theoretically optimal algorithm should simply output a hypergraph with maximum posterior likelihood, i.e., following the maximum \emph{a posteriori} (MAP) rule. As discussed next, the MAP rule can be easily characterized.

\paragraph{MAP Outputs a Minimum Preimage.}

Since the posterior distribution is
\[
p_{\rhG|\pG}(H|\pG) =\frac{\ind{\proj(\hE_H) = \pE}p_{\rhG}(\hE_H) }{p_\pG(\pE)}\propto  \ind{\proj(\hE_H) = \pE} \b(\frac{p}{1-p}\b)^{|\hE_H|}\,,
\]
the optimal algorithm $\cA^*$ should output one of the hypergraphs that project to $\pG$ with the smallest number of hyperedges, i.e., 
\[
\cA^*(\pG) \in \argmin_{\hG:\hE_{\hG}\in \prim(\pE)} |\hE_{\hG}|\,.
\]We say a hypergraph $\hG$ is 
a \emph{minimum preimage} if $\hG\in \argmin_{\hE_\hG\in \prim(E)} |\hE_{\hG}|$.

Since ties can be broken arbitrarily, we always assume that $\cA^*$ chooses a specific minimum preimage (for instance based on lexicographical order on the hyperedges) instead of choosing a random one.



\subsection{MAP is Efficient for Sparse Graphs}

In general, the MAP algorithm involves solving for the minimum way to cover a graph with a hypergraph, which can be intractable. In this section, we will provide an efficient algorithm for computing the MAP rule if the hypergraph is sparse enough, of course also making use of the fact that the hypergraph is random.
\begin{theorem}\label{thm:sparse-algo}
    When $\delta<\frac{d-1}{d+1}$, the optimal algorithm $\cA^*$ is with high probability efficiently computable (i.e., has runtime polynomial in $n$).
\end{theorem}
The underlying intuition is that when the hypergraph is sparse enough, we can partition the projected graph into constant-size components, where the minimum preimage of each component can be solved for independently of the other components. However, a naive definition of connected component is useless, as $p$ is far above the connectivity threshold. We require a definition of component better suited to our goal of finding the minimum preimage.

\paragraph{2-Neighborhood and 2-Connectivity.}

Define the \emph{2-neighbor} of a hyperedge $\he$ in a hypergraph $\hG$ to be all hyperedges $\he'$ with $|\he\cap \he'|\ge 2$, denoted by 
\[
\nei{\hG}(\he) = \{\he':|\he\cap \he'|\ge 2\}\,.
\]
Let $\mathcal{G}_\hG$ be a graph whose node set is the set of hyperedges in $\hG$ and the neighborhood structure is defined by 2-neighbors. We say that a set of hyperedges in $\hG$ is \emph{2-connected} if they are connected in $\mathcal{G}_\hG$. A \emph{2-connected component} of the hypergraph $\hG$ is a set of hyperedges that form a connected component in $\mathcal{G}_\hG$.\footnote{Here the definition is for hypergraphs and is different from the usual definition of 2-connectivity in a simple graph. }
See an illustration of 2-connectivity in Figure~\ref{fig:2-connectivity}. 
We will never need to refer to the graph $\mathcal{G}_\hG$ and instead work directly with 2-connected sets of hyperedges in $H$. 

\begin{figure}
     \centering
     \begin{subfigure}[t]{0.48\textwidth}
         \centering
         \begin{tikzpicture}[every node/.style={fill,circle,inner sep=1.5pt},yscale=0.9,xscale=0.8]

  \coordinate (P1) at (0,0);
  \coordinate (P2) at (2,0);
  \coordinate (P3) at (1,1.732); 
  \coordinate (P4) at (3,1.732);

  \filldraw[orange, fill=orange!20] (P1) -- (P2) -- (P3) -- cycle; 
  \filldraw[orange, fill=orange!20] (P2) -- (P3) -- (P4) -- cycle; 
  \coordinate (P5) at (4,0); 
  \filldraw[orange, fill=orange!20] (P2) -- (P4) -- (P5) -- cycle; 
  \coordinate (P6) at (5,1.732); 
  \filldraw[orange, fill=orange!20] (P4) -- (P5) -- (P6) -- cycle; 
  \coordinate (P7) at (1,-1.732);
  \filldraw[orange, fill=orange!20] (P1) -- (P2) --(P7) -- cycle;

  \coordinate (R1) at (5,-1.732);
  \coordinate (R2) at (6,0); 
  \coordinate (R3) at (7,-1.732);

  \filldraw[red, fill=red!20] (P5) -- (R1) -- (R2) -- cycle; 
  \filldraw[red, fill=red!20] (R2) -- (R1) -- (R3) -- cycle; 

  \coordinate (R4) at (8,0);
  \filldraw[red, fill=red!20] (R2) -- (R3) -- (R4) -- cycle;

  \coordinate (C1) at (8,0); 
  \coordinate (C2) at (9,1);
  \coordinate (C3) at (9,-1);
  \coordinate (C4) at (10,0);
  \filldraw[blue, fill=blue!20] (R4) -- (C1) -- (C2) -- (C4) -- (C3) -- (C1); 
  \draw[blue] (C2) -- (C3); 
  \draw[blue] (C4) -- (C1); 

  
\end{tikzpicture}
         \caption{An example of hypergraph $H$. Triangles filled with colors represent hyperedges in $\hG$. Two hyperedges are 2-connected if they share two nodes. Different 2-connected components are marked in different colors. }
     \end{subfigure}
     \hfill
     \begin{subfigure}[t]{0.48\textwidth}
         \centering
\begin{tikzpicture}[,yscale=0.9,xscale=0.8]
  \coordinate (P1) at (0,0);
  \coordinate (P2) at (2,0);
  \coordinate (P3) at (1,1.732); 
  \coordinate (P4) at (3,1.732);
  \coordinate (P5) at (4,0);
  \coordinate (P6) at (5,1.732);
  \coordinate (P7) at (1,-1.732);

  \coordinate (R1) at (5,-1.732);
  \coordinate (R2) at (6,0); 
  \coordinate (R3) at (7,-1.732);

  \coordinate (C1) at (8,0); 
  \coordinate (C2) at (9,1);
  \coordinate (C3) at (9,-1);
  \coordinate (C4) at (10,0);
  \coordinate (C5) at (9,0);

  \node (b1) at (barycentric cs:P1=1,P2=1,P3=1) [circle, fill=orange, inner sep=1.5pt] {};
  \node (b2) at (barycentric cs:P2=1,P3=1,P4=1) [circle, fill=orange, inner sep=1.5pt] {};
  \node (b3) at (barycentric cs:P2=1,P4=1,P5=1) [circle, fill=orange, inner sep=1.5pt] {};
  \node (b4) at (barycentric cs:P4=1,P5=1,P6=1) [circle, fill=orange, inner sep=1.5pt] {};
  \node (b5) at (barycentric cs:P1=1,P2=1,P7=1) [circle, fill=orange, inner sep=1.5pt] {};

  \draw[orange] (b1) -- (b2);
  \draw[orange] (b2) -- (b3);
  \draw[orange] (b3) -- (b4);
  \draw[orange] (b1) -- (b5);

  \node (br1) at (barycentric cs:P5=1,R1=1,R2=1) [circle, fill=red, inner sep=1.5pt] {};
  \node (br2) at (barycentric cs:R2=1,R1=1,R3=1) [circle, fill=red, inner sep=1.5pt] {};
  \node (br3) at (barycentric cs:R2=1,R3=1,C1=1) [circle, fill=red, inner sep=1.5pt] {};
  \draw[red] (br1) -- (br2) -- (br3);

    \node (bc1) at (barycentric cs:C1=1,C2=1,C5=1) [circle, fill=blue, inner sep=1.5pt] {};
  \node (bc2) at (barycentric cs:C2=1,C4=1,C5=1) [circle, fill=blue, inner sep=1.5pt] {};
  \node (bc3) at (barycentric cs:C3=1,C4=1,C5=1) [circle, fill=blue, inner sep=1.5pt] {};
  \node (bc4) at (barycentric cs:C3=1,C1=1,C5=1) [circle, fill=blue, inner sep=1.5pt] {};
  \draw[blue] (bc1) -- (bc2) -- (bc3) -- (bc4) -- (bc1);
\end{tikzpicture}
         \caption{The corresponding graph $\cG_H$. Each node in $\cG_H$ corresponds to a hyperedge in $H$. Two nodes are connected if their corresponding hyperedges share two nodes.}
         \label{fig:three sin x}
     \end{subfigure}
\caption{An example of  2-connectivity and 2-connected components when $d=3$.}
\label{fig:2-connectivity}
\end{figure}
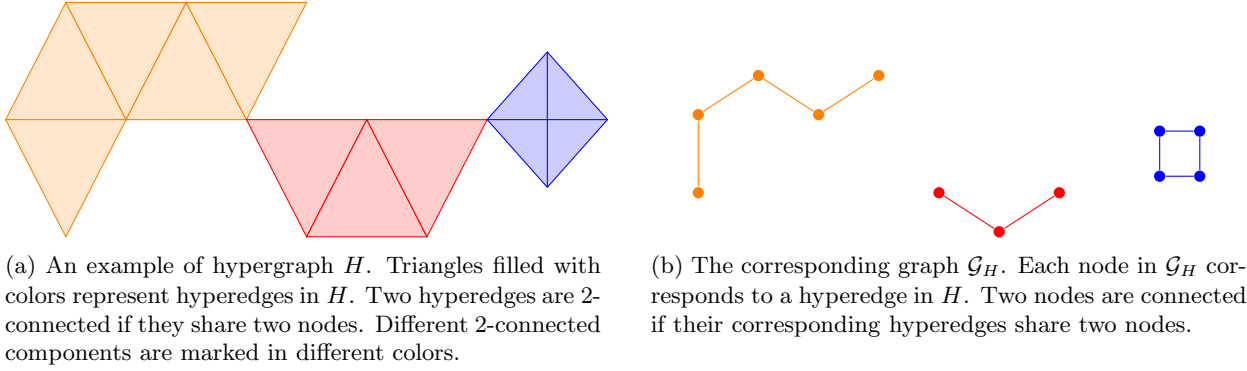

\subsubsection{Decomposition of MAP Across 2-Connected Components}

The following lemma implies that the task of finding the minimum preimage \emph{decomposes} and can be carried out individually in each of the 2-connected components of hyperedges.

Recall that the clique hypergraph $\cliG$ (defined at the start of Section~\ref{sec:maxclique}) of a graph $G=(V,E)$ has hyperedge set $\cli(E) = \b\{\he\in \textstyle{\binom{[n]}{d}} : (i,j)\in E
\text{ for every } \{i,j\}\subset h
\b\}$.

\begin{lemma}\label{lem:union-min-preimage}
Let $C_1,\cdots, C_m$ be all 2-connected components (i.e., 2-connected subsets of hyperedges) of the clique hypergraph $\cliG$ of the projected graph $G_p=([n],E_p)$. We have
\[
\prim(\pG) = \{\cup_{i=1}^m\hG_i: \hG_i\in \prim(\proj(C_i))\}.
\]
In words, any preimage of $\pG$ is given by a union of hypergraphs, each from a preimage of the projection of a 2-connected component of $\cliG$.

\end{lemma}
The proof of the lemma is given in Appendix~\ref{sec:union-min-preimage}. 

What makes this decomposition so useful is that with high probability each of the components is of constant size. This will allow us to carry out a brute-force search on each component. 

\begin{restatable}{lemma}{component}\label{lem:component-constant-size}
     For any fixed $\delta<\frac{d-1}{d+1}
     $, with high probability, all 2-connected components of $\cliG$ have size at most $1+2^{d+1}/(\frac{d-1}{d+1}-\delta)=O_n(1)$.
\end{restatable}

We will refer to this threshold, $\tcth$, as the \emph{\cth}. 

\subsubsection{MAP Algorithm}
We have the following efficient algorithm that (with high probability) implements the MAP rule $\cA^*$:
\begin{algorithm}\caption{Maximum a Posteriori (MAP) $\cA^*$}\label{alg:map}
\begin{algorithmic}[1]
\State Input: $\pG = ([n],\pE)$
\State Calculate the clique graph $\cliG$ from $\pG$ by finding all size-$d$ cliques in $\pG$.
\State Enumerate over all pairs of vertices to determine 2-neighborhoods of all hyperedges in $\cliG$.
\State Find all 2-connected components of $\cliG$ using a depth-first search on all hyperedges in $\cliG$.
\State Search over all preimages in each 2-connected components of $\cliG$ and find one with minimum size. 
\State Output the union of the minimum preimages of the 2-connected components in $\cliG$.
\end{algorithmic}
\end{algorithm}



\begin{proof}[Proof of Theorem~\ref{thm:sparse-algo}]
From the previous section, we know that MAP returns an arbitrary minimum preimage of the projected graph $\pG$. By Lemma~\ref{lem:union-min-preimage}, a minimum preimage of $\pG$ is given by the union of minimum preimages of all 2-connected components in $\cliG$. So Algorithnm~\ref{alg:map} indeed implements the MAP rule.

We now analyze the running time of the algorithm. Steps 2 and 4 take time at most $O_n(n^d)$. Step 3 takes time at most $O_n(n^2)$. 
Step 5 takes time at most $n^d2^k$, where $k$ is the size of the largest 2-connected component in $\cliG$. By Lemma~\ref{lem:component-constant-size}, $k=O_n(1)$ with high probability, so overall the algorithm finishes in time $O_n(n^d)$ with high probability.
\end{proof}

\subsubsection{2-Connected Components have Constant Size for Sparse Hypergraphs}\label{sec:const-size}
In this section we give a proof sketch of Lemma~\ref{lem:component-constant-size} which states that $\cliG$ can be partitioned into small 2-connected components for $\delta$ below $\frac{d-1}{d+1}$. 
We give the full proof in Section~\ref{sec:growth-components}.

The lemma is proved by carefully examining how a set of 2-connected edges in $\cliG$ can grow bigger. This is analogous to (but  more delicate than) the analysis of components in subcritical Erd\H os-R\'enyi graphs.  We will show that any 2-connected component can be decomposed into a series of ``growth'' steps starting from a single hyperedge. Each growth operation has a ``probabilistic cost" because it is a moderately low probability event, which reduces the number of such components. Accounting for the possible growth patterns within 2-connected components in $\cliG$ shows that with high probability no large components appear. 

Now let us consider the possible ways to grow a sub-hypergraph $\shG\subset \rhG$ via local exploration, and try to understand why the probability of having the graph in $\cliG$ decreases with the growth. Suppose $\shG$ is a set of hyperedges, and $\cli(\proj(\shG))$ is 2-connected. For $\shG$ to get larger, it must include one of its 2-neighbors $\he\in \cliG$. How did $\he$ appear in $\cliG$? The somewhat delicate aspect of this is that $\he$ may not be in $\rhG$: $\he$ might exist in $\cliG$ because all edges in the clique $\proj(\he)$ are covered by \emph{some other} hyperedges $\hE\subset \rhG$. So to grow $\shG$, one option is to include all of $\hE$. 
Because each hyperedge is included with fairly small probability, this reduces the expected number of components of the given form, while the number of options in selecting $\hE$ increases with the size of $\hE$. The following lemma gives
the expected number of appearances of a given sub-hypergraph $\shG$ in terms of the number of nodes and the number of hyperedges in the sub-hypergraph.  



\begin{lemma}\label{lem:number-appearance} Let $X_{\shG}$ be the number of appearances of a sub-hypergraph $\shG$ in $\rhG$. Denote by $v_\shG$ and $e_\shG$ the number of nodes and hyperedges in $\shG$.
    For any hypergraph $\shG$, 
    \[
    \E X_{\shG} = \Theta_n(n^{v_\shG}p^{e_\shG})\,.
    \]
\end{lemma}
When we grow $\shG$, we increase both the number of nodes and the number of edges of the hypergraph. With more nodes, the expectation increases (more possible choices) and with more hyperedges, the expectation decreases. The trade-off is controlled by how we choose $\hE$ and the parameter $\delta$. When $\delta<\frac{d-1}{d+1}$, we will be able to show that no matter how $\hE$ is chosen, the expectation always decreases by a polynomial factor. Therefore, after a constant number of growth steps, the expectation becomes negligible.


\subsection{Ambiguous Graphs and Success Probability of MAP}
In this section, we will see that when $\delta $ is below the \cth\  $\tcth$, the success probability of MAP is fully determined by graphs with non-unique minimum preimages, which we call ambiguous graphs.

\begin{definition}\label{def:ambiguous}
    An \emph{ambiguous graph} is a graph with at least two minimum hypergraph preimages.
\end{definition}

As we will see, appearance or non-appearance of ambiguous graphs determines success of the MAP rule. 

\subsubsection{Impossibility Result via Existence 
 of Ambiguous Graphs}
 
In the previous section, we showed that MAP is w.h.p. efficient whenever $\delta<\frac{d-1}{d+1}$. However, even the optimal algorithm does not always succeed in this regime. Consider the graph depicted in Figure~\ref{fig:ambiguous-graph}. If $\pG$ has a copy of this graph as a component, then there are two minimum preimages with equal size, both with the same posterior probability. So no matter which one the MAP algorithm outputs, it must incur at least 1/2 probability of error. This is formalized in the following lemma.

\begin{figure}

   \begin{minipage}{0.48\textwidth}\centering
\begin{tikzpicture}[node/.style={minimum size = 1mm, inner sep=0pt}, scale=0.7]
    \coordinate (1) at (-1.73,0);
    \coordinate (2) at (0,1);
    \coordinate (3) at (0,-1);
    \coordinate (4) at (1.73,0);
    \coordinate (5) at (-2,2);
    \coordinate (6) at (-2,-2);
    \coordinate (7) at (2,-2);
    \coordinate (8) at (2,2);

    \fill[green!30] (1) -- (2) -- (3) -- cycle;
    \fill[green!30] (1) -- (2) -- (5) -- cycle;
    \fill[green!30] (1) -- (3) -- (6) -- cycle;
    \fill[green!30] (4) -- (2) -- (8) -- cycle;
    \fill[green!30] (4) -- (7) -- (3) -- cycle;

    \draw (1) -- (2) -- (3) -- cycle;
    \draw (2) -- (3) -- (4) -- cycle;
    \draw (1) -- (2) -- (5) -- cycle;
    \draw (1) -- (3) -- (6) -- cycle;
    \draw (2) -- (4) -- (8) -- cycle;
    \draw (3) -- (4) -- (7) -- cycle;

    \node at (1) [draw, circle, fill=black, scale=0.5,inner sep=2pt, label=left:1] {};
    \node at (2) [draw, circle, fill=black, scale=0.5,inner sep=2pt, label=above:2] {};
    \node at (3) [draw, circle, fill=black, scale=0.5,inner sep=2pt, label=below:3] {};
    \node at (4) [draw, circle, fill=black, scale=0.5,inner sep=2pt, label=right:4] {};
    \node at (5) [draw, circle, fill=black, scale=0.5,inner sep=2pt, label=above:5] {};
    \node at (6) [draw, circle, fill=black,scale=0.5, inner sep=2pt, label=below:6] {};
    \node at (7) [draw, circle, fill=black, scale=0.5,inner sep=2pt, label=below:7] {};
    \node at (8) [draw, circle, fill=black, scale=0.5,inner sep=2pt, label=above:8] {};

\end{tikzpicture}
\end{minipage}\hfill
   \begin{minipage}{0.48\textwidth}\centering

\begin{tikzpicture}[node/.style={minimum size = 0mm, inner sep=0pt}, scale=0.7]
    \coordinate (1) at (-1.73,0);
    \coordinate (2) at (0,1);
    \coordinate (3) at (0,-1);
    \coordinate (4) at (1.73,0);
    \coordinate (5) at (-2,2);
    \coordinate (6) at (-2,-2);
    \coordinate (7) at (2,-2);
    \coordinate (8) at (2,2);

    \fill[green!30] (4) -- (2) -- (3) -- cycle;
    \fill[green!30] (1) -- (2) -- (5) -- cycle;
    \fill[green!30] (1) -- (3) -- (6) -- cycle;
    \fill[green!30] (4) -- (2) -- (8) -- cycle;
    \fill[green!30] (4) -- (7) -- (3) -- cycle;

    \draw (1) -- (2) -- (3) -- cycle;
    \draw (2) -- (3) -- (4) -- cycle;
    \draw (1) -- (2) -- (5) -- cycle;
    \draw (1) -- (3) -- (6) -- cycle;
    \draw (2) -- (4) -- (8) -- cycle;
    \draw (3) -- (4) -- (7) -- cycle;

    \node at (1) [draw, circle, fill=black,scale=0.5, inner sep=2pt, label=left:1] {};
    \node at (2) [draw, circle, fill=black,scale=0.5, inner sep=2pt, label=above:2] {};
    \node at (3) [draw, circle, fill=black,scale=0.5, inner sep=2pt, label=below:3] {};
    \node at (4) [draw, circle, fill=black,scale=0.5, inner sep=2pt, label=right:4] {};
    \node at (5) [draw, circle, fill=black,scale=0.5, inner sep=2pt, label=above:5] {};
    \node at (6) [draw, circle, fill=black,scale=0.5, inner sep=2pt, label=below:6] {};
    \node at (7) [draw, circle, fill=black,scale=0.5, inner sep=2pt, label=below:7] {};
    \node at (8) [draw, circle, fill=black,scale=0.5, inner sep=2pt, label=above:8] {};

\end{tikzpicture}

\end{minipage}
\caption{A graph with non-unique minimum preimage in the case $d=3$. The green hyperedges are the two possible minimum preimages.}
\label{fig:ambiguous-graph}
\end{figure}
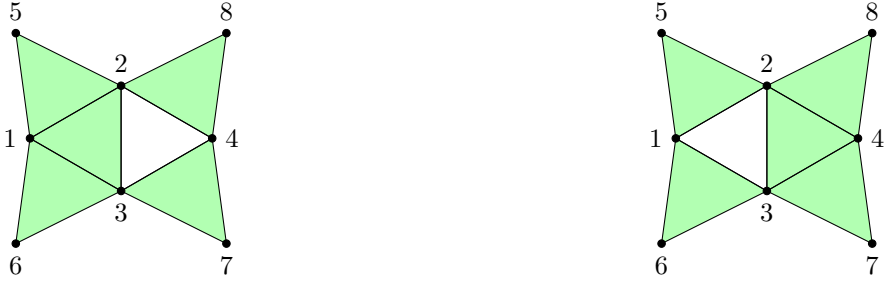

\begin{restatable}{lemma}{badgraphinisolation}\label{lem:bad-graph-in-isolation}
    For any ambiguous graph $G_a$ and any recovery algorithm $\cA$, given input $\pG = \proj(\rhG)$,
    \[
    \Pr(\cA(\pG)\ne\rhG) \ge \frac{1}{2}\Pr\b(\cli(G_a)\text{ is a 2-connected component of }\cliG \b)\,.
    \]
\end{restatable}


\begin{proof}
By Lemma~\ref{lem:union-min-preimage}, a minimum preimage of $\pG$ is given by the union of the minimum preimages of every 2-connected component of $\cliG$. Therefore, when $\cli(G_a)$ is a 2-connected component of $\cliG$, the minimum preimage of $\cliG$ is not unique. So no matter which hypergraph $\cA^*$ chooses, it has at least 1/2 probability of making a mistake. In other words,
\[
\Pr(\cA(\pG)\ne\rhG|\cli(G_a)\text{ is a 2-connected component of }\cliG) \ge 1/2\,.
\]
The lemma follows from Bayes rule. 
\end{proof}

It follows that to prove impossibility of (exact) recovery, we only need to find an ambiguous graph that is a 2-connected component with probability $\Omega_n(1)$.
Let the \emph{ambiguity threshold}, $\ath$, be the infimum of $\delta$ such that there exists an ambiguous graph appearing as a 2-connected component with probability $\Omega_n(1).$
\[
\ath \defeq \inf\{\delta:\exists G_a, \Pr(\cli(G_a)\text{ is a 2-connected component of }\cliG) = \Omega_n(1) \}\,.
\]
It then follows from Lemma~\ref{lem:bad-graph-in-isolation} that exact recovery is impossible for any $\delta$ that is at least $\ath$. 
In other words, we have the following corollary.
\begin{corollary}\label{cor:critical-threshold-smaller-ambiguous-threshold}
    For any $d$, we have $\cdel d \le \ath\,.$
\end{corollary}

This will allow us to prove the impossibility results in Theorem~\ref{thm:main-three} and Theorem~\ref{thm:main-four-five} showing that $\ath$, and hence also $\cdel d$, is at most $\frac{2d-4}{2d-1}$ when $d\le 5$.  The construction of the ambiguous graph will be described in Section~\ref{sec:impossibility}. It will be a generalization of Figure~\ref{fig:ambiguous-graph} to general $d$.

This approach stops working for $d\ge 6$. In Section~\ref{sec:impossibility}, we will show that the regime of $\delta$ in which such an ambiguous graph appears as a 2-connected component in $\rhG$ is between $\frac{2d-4}{2d-1}$ and the \cth\ $\frac{d-1}{d+1}$. When $d\ge 6$, $\frac{2d-4}{2d-1}$ is above the \cth\ and the ambiguous graph typically overlaps with other hyperedges, i.e., it does not appear as a component. In this case it is no longer clear that there are at least two equally likely preimages.

We next work towards understanding
when a given sub-hypergraph will appear in $\rhG$.

\subsubsection{Appearance of Sub-hypergraphs in Random Hypergraphs}

We will need a lemma that determines the threshold density for a given graph to appear in the random $d$-hypergraph, $\rhG(n,d,p)$.
The graph version of the lemma was first proven in \cite{bollobas1981threshold} and simplified in \cite{rucinski1986strongly}. For random hypergraphs, the proof is similar and we include it in Appendix~\ref{sec:subgraph_threshold} for completeness.
\begin{lemma}\label{lem:subgraph_threshold}
    For a hypergraph $\shG = (V,\hE_\shG)$,
    define \[
    m(\shG) = \max_{\shG'\subset \shG}\frac{e_{\shG'}}{v_{\shG'}}\,,
    \]
    where $e_{\shG'}$ and $v_{\shG'}$ are the number of edges and the number of nodes of sub-hypergraph $\shG$.
    We have
    \[
    \Pr(\shG\subset \rhG) = 
    \begin{cases}
        o_n(1) &\text{if }p=o_n(n^{-1/m(\shG)})\\
        1-o_n(1) &\text{if }p=\omega_n(n^{-1/m(\shG)})\\
        \Omega_n(1) &\text{if } p =\Theta_n(n^{-1/m(\shG)}).
    \end{cases}
    \]
\end{lemma}

\subsubsection{Reconstruction Result by Nonexistence of Ambiguous Graphs}\label{sec:main-idea-reconstruction}
In this section we prove the following theorem.
\begin{theorem}\label{thm:delta-lower}
    When $d=3$ and $\delta<2/5$ or when $d=4,5$ and $\delta<1/2$, the MAP rule achieves exact recovery and moreover it can be implemented efficiently.
\end{theorem}
In the regime where $\delta<
\frac{d-1}{d+1}$, which is the regime we care about when $d\le 5$, the converse of Lemma~\ref{lem:bad-graph-in-isolation} is also true. That is, if with high probability no ambiguous graph (i.e., with non-unique minimum cover) appears in $\pG$ as a 2-connected component, then MAP succeeds with high probability.

\begin{lemma}\label{lem:no-ambiguous}
    Assume $\delta<
    \frac{d-1}{d+1}$.
    If for all finite ambiguous graphs $G_a$,
    \[
    \Pr\b(\cli(G_a)\text{ is a 2-connected component of }\cliG \b)=o_n(1)\,,
    \]
    then we have 
    \[
    \Pr(\cA^*(\pG)=\rhG)\ge 1- o_n(1)\,.
    \]
\end{lemma}

The lemma is proved in Appendix~\ref{sec:proof-no-ambiguous}. 
Here we provide a sketch of the proof. If there is no ambiguous graph in $\pG$, the projections of every 2-connected components have a unique minimum preimage. As shown in Lemma~\ref{lem:component-constant-size}, all 2-connected components are of constant size. Under this condition, the minimum preimage of the 2-connected component is correct with probability $1-O_n(p)$, as any other preimage is $O_n(p)$ times less likely in the posterior and there are only constant number of possible preimages. The overall minimum preimage, as given by $\cA^*$, is then correct with high probability by union bound over all 2-connected components.

Recall the definition of the ambiguous threshold, $\ath$, Lemma~\ref{lem:no-ambiguous} implies that the critical threshold $\cdel d$ is above $\ath$ if $\ath$ is below $\tcth$.
\begin{corollary}
    For any $d$, if $\ath\le\tcth$, we have $ \cdel d \ge \ath$.
\end{corollary}

Combining this corollary with Corollary~\ref{cor:critical-threshold-smaller-ambiguous-threshold}, we get that the ambiguous threshold $\ath$ fully determines the critical threshold $\cdel d$ if $\ath$ is below $\tcth$.
\begin{corollary}
    For any $d= 3,4,5$, we have $\ath\le\tcth$, and hence $ \cdel d = \ath$.
\end{corollary}
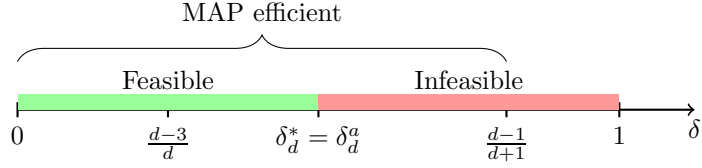
\begin{figure}
    \centering
    \begin{tikzpicture}
    \draw[thick, ->] (-4,0) -- (5,0) node[anchor=north] {$\delta$};

    \draw[thick] ({-2},0.1) -- ({-2},-0.1) node[below] {$\frac{d-3}{d}$};

    \draw[thick] (0,0.1) -- (0,-0.1) node[below] {$\cdel d = \ath$};

    \draw[thick] (2.5,0.1) -- (2.5,-0.1) node[below] {$
    \frac{d-1}{d+1}$};

    \foreach \x in {0,1}
        \draw[thick] (8*\x-4,0.1) -- (8*\x-4,-0.1) node[below] {\x};

    \fill[green!40] (-4,0) rectangle (0.4,0.2);
    \node at (-2,0.4) {Feasible};

    \fill[red!40] (0,0) rectangle (4,0.2);
    \node at (2,0.4) {Infeasible};

    \draw [decorate,decoration={brace,amplitude=10pt,raise=4pt},yshift=0pt]
    (-4,0.5) -- (2.5,0.5) node [black,midway,yshift=0.8cm] {MAP efficient};
\end{tikzpicture}
    \caption{
    Relation between different thresholds. The maximum clique cover algorithm $\cliA$ succeeds with high probability up to $\delta=\frac{d-3}{d}$. The MAP algorithm is efficient up to $\tcth$ and succeeds with high probability up to threshold $\cdel d$. If $\ath<\tcth$, then $\cdel d$ is the same as the ambiguous threshold $\ath$.
    }
    \label{fig:thresholds-relation}
\end{figure}

As long as we can check the condition in Lemma~\ref{lem:no-ambiguous} for a specific $\delta$, MAP is optimal. If $\delta<\tcth$, then with high probability all 2-connected components have size bounded by $(2^d+1)/(\frac{d-1}{d+1}-\delta)$, so there are only finitely many graphs we need to check. This gives us the following computer assisted method of proving that MAP works when $\delta $ is below a hypothesized threshold $\delta_0$:
\begin{enumerate}
    \item Enumerate over all hypergraphs $\shG$ with at most $1+2^{d+1}/(\frac{d-1}{d+1}-\delta_0)$ hyperedges.
    \item Compute the probability that $\shG\subset \rhG$ by Lemma~\ref{lem:subgraph_threshold} with $p=n^{-d+1+\delta_0}$.
    \item Enumerate all possible preimages of $\proj(\shG)$ and see if $\proj(\shG)$ is ambiguous.
    \item If all graphs are either not ambiguous or have vanishing probability of occurring, the condition in Lemma~\ref{lem:no-ambiguous} is satisfied and MAP succeeds with high probability at $\delta=\delta_0$.
\end{enumerate}
Since $\Pr(\shG\subset \rhG)$ monotonically increases with $\delta$, we know the same condition holds for any $\delta<\delta_0$.

Although this approach can be carried out in principle, the number of hypergraphs with at most $1+2^{d+1}/(\frac{d-1}{d+1}-\delta_0)$ hyperedges is a huge number and cannot be verified in reasonable time. Instead of doing a brute force search, we will utilize the structure of how 2-connected components grow, as discussed in Section~\ref{sec:const-size}, to reduce the runtime. The runtime of the search can be further reduced by identifying properties of ambiguous graph and focusing on graphs with such properties. 

With the computer search, we are able to prove the following lemma. The search algorithm will be discussed in more detail in Section~\ref{sec:algo-proof}.
\begin{restatable}{lemma}{lowerbound}\label{lem:delta-lowerbound}
    When $d=3$ and $\delta<2/5$ or when $d=4,5$ and $\delta<1/2$, any ambiguous graph $G_a$ 
    satisfies
    $
    \Pr(G_a\subset \pG)=o_n(1).
    $
\end{restatable}

Combining this lemma and Lemma~\ref{lem:no-ambiguous} completes the proof of Theorem~\ref{thm:delta-lower}.
\subsection{Upper Bound on $\cdel d$ for Large $d$ and Proof of Theorem~\ref{thm:main-large-d}}
We identify a sufficient condition for the (optimal) MAP rule to fail:
Suppose there is a hyperedge $\he$ in $\rhG$ where every pair of nodes in $\he$ is also included in other hyperedges in $\rhG$. In this case the graph $\rhG\setminus \{\he\}$ has higher probability and has the same graph projection. 
Because the optimal algorithm outputs a minimum preimage, it does not output the original hypergraph $\rhG$: deleting $\he$ forms a smaller preimage.

We formalize this sufficient condition and consider a hypergraph $\bG$ with the following hyperedges:
\begin{itemize}
    \item $\{v_1,\cdots,v_d\}$,
    \item $\{v_i,v_j,u_{ij}^{(1)}, u_{ij}^{(2)},\cdots, u_{ij}^{(d-2)}\}$ for all $\{i,j\}\subset [d] $, where for each $i$ and $j$ the nodes $u_{ij}^{(1)}, u_{ij}^{(2)},\cdots, u_{ij}^{(d-2)}$ are arbitrary.
\end{itemize}

From the discussion above, we know that $\cA^*$ will fail if $\bG\subset\rhG$, because the hyperedge $v_1,\cdots,v_d$ can be removed from the output and increase the posterior probability. 
Therefore, we have
\[
\Pr(\cA^*(\pG)\ne \rhG) \ge \Pr(\bG\subset\rhG)\,.
\]
By Lemma~\ref{lem:subgraph_threshold}, this occurs with probability $\Omega_n(1)$ when $p=\Omega_n(n^{-1/m(\bG)})$. Since 
\[
m(\bG) = \frac{e_{\bG}}{v_{\bG}} = \frac{\binom{d}{2}+1}{d+\binom{d}{2}(d-2)}\,,
\]
this is equivalent to 
\[
p=\Omega_n(n^{-d\frac{d^2-3d+4}{d^2-d+2}})\,,
\]
or $\delta\ge \frac{d^2-d-2}{d^2-d+2}$.
We have shown the following impossibility result:
\begin{theorem}\label{thm:large-d}
    Exact recovery is information theoretically impossible when  $\delta\ge \frac{d^2-d-2}{d^2-d+2}$\,.
\end{theorem}


%% file: NaiveAlgorithms.tex
\subsection{Maximum Clique Cover Algorithm}
\label{sec:maxclique}
When the graph is so sparse that each hyperedge appears as an isolated clique, exact recovery is easily achieved by creating a hypergraph with a hyperedge for every clique of the projected graph $\pG$. This algorithm turns out to succeed far beyond the regime where hyperedges do not overlap.

Let the \emph{$d$-clique hypergraph }$\cliG$ of the projected graph $\pG=([n], E_p)$ be the hypergraph $\cliG = ([n],\cliE=\cli(\pE))$ where
\[
\cli(E) = \b\{\he\in \textstyle{\binom{[n]}{d}} : (i,j)\in E
\text{ for every } \{i,j\}\subset h
\b\}\,.
\]
%
Denote by $\cliA$ the algorithm converting every size-$d$ clique in $\pG $ to a hyperedge in the output graph, i.e., $\cliA(\pG) = \cli(\pE)$. We call this the \emph{maximum clique cover algorithm}.


\begin{algorithm}\caption{Maximum Clique Cover Algorithm $\cliA$}\label{alg:clique-cover}
\begin{algorithmic}[1]
\State Input: $\pG = ([n],\pE)$
\State  $\cli(\pE)\leftarrow \emptyset$
\For {all size $d$ subsets of $[n]$}
    \State If $\pE$ has a clique on the subset, add the hyperedge on the subset to $\cli(\pE)$
\EndFor
\State Output $\cli(\pE)$
\end{algorithmic}
\end{algorithm}

\begin{remark}
    Since we are enumerating all size-$d$ subsets, the algorithm has time complexity $n^d$. It may be possible to improve this runtime by taking advantage of sparsity of the graph, using ideas in \cite{boix2021average}. 
\end{remark}

For which parameters does this algorithm work? 
From the definition, $\cliA$ fails if and only if there exists a clique in $\pG$ that is not a hyperedge of $\rhG$. 
If a $d$-clique $\he$ in $\pG$ is not a hyperedge of $\rhG$, every edge in the clique is included in some other hyperedge $\he'\in \rhG$. By carefully examining the possible ways of inclusion for all edges, we can obtain a tight bound on the probability of the event, yielding the following threshold.

\begin{theorem}\label{thm:naive}
    $\cliA$ exactly recovers $\rhG$ when $\delta< \frac{d-3}{d}$ and has $\Omega_n(1)$ probability of failure when $\delta\ge \frac{d-3}{d}$.
\end{theorem}
This implies the positive recovery  result in Theorem~\ref{thm:main-large-d} for $d\geq 6$, which we believe to be suboptimal. The proof of Theorem~\ref{thm:naive} is in Appendix~\ref{sec:naive}. 

\subsection{Greedy Algorithm}
Another natural algorithm starts with the maximum clique cover algorithm and then greedily deletes redundant hyperedges from the clique graph.

\begin{algorithm}\caption{Greedy Algorithm}\label{alg:greedy}
\begin{algorithmic}[1]
\State Input: $\pG = ([n],\pE)$
\State Find the $d$-clique hypergraph $H_0\leftarrow\cli(\pE)$
\While{$\exists h\in H_0$ that $H_0\backslash h\in \proj^{-1}(\hE_\rhG)$}
    \State $H_0\leftarrow H_0\backslash h$
\EndWhile
\State Output $H_0$
\end{algorithmic}
\end{algorithm}



Heuristically this algorithm ought to work better than the maximum clique cover algorithm, because it yields a graph with higher posterior probability. We leave it as an open question to determine under which parameter regime this algorithm succeeds. 

%% file: Upperbound.tex
\section{Impossibility when $d\le 5$ and $\delta \ge \frac{2d-4}{2d-1}$}\label{sec:impossibility}

Recall that in Lemma~\ref{lem:bad-graph-in-isolation}, we reduced the problem of proving impossibility of exact recovery to finding ambiguous graphs. 
\badgraphinisolation*
In this section we will identify an appropriate ambiguous graph and then use Lemma~\ref{lem:bad-graph-in-isolation} to prove the following theorem.
\begin{theorem}\label{thm:upper}
    When $d=3,4,5$, for any $\delta\ge \frac{2d-4}{2d-1}$, exact recovery is information theoretically impossible.
\end{theorem}
\subsection{Ambiguous Graph and Its Properties}

Let us list the properties we need for an ambiguous graph $\aG$ to prove the theorem:
\begin{enumerate}
    \item The graph should be ambiguous, i.e., it should have at least two minimum preimages. We will prove this property in Lemma~\ref{lem:bad-graph-non-unique-preimage}. 
    \item The graph appears in $\pG$ with constant probability, $\Pr(\aG\subset \pG)=\Omega_n(1)$. We will prove this  property in Lemma~\ref{lem:bad-graph-contain-probability}.
    \item The graph appears as a 2-connected component in $\cliG$ with constant probability.
    We will prove this property in Corollary~\ref{cor:bad-graph-probability}.
\end{enumerate}

Recall that ambiguous graph was defined in Defn.~\ref{def:ambiguous}. We will construct a specific such graph.

\begin{definition}[Ambiguous Graph $\aG$]\label{def:aG}
   We define the graph $\aG$ as the union of the following $2d$ cliques:
\begin{itemize}
    \item the clique $u_1,v_1,v_2,\cdots, v_{d-1}$, denoted by $\he_1$,
    \item the clique $u_2,v_1,v_2,\cdots, v_{d-1}$, denoted by $\he_2$,
    \item for any $1\le i\le d-1$, the clique $u_1,v_i,w_i^{(1)},w_i^{(2)},\cdots, w_i^{(d-2)}$, denoted by $\he_i^w$,
    \item and for any $1\le i\le d-1$, the clique $u_1,v_i,z_i^{(1)},z_i^{(2)},\cdots, z_i^{(d-2)}$, denoted by $\he_i^z$.
\end{itemize}
Let $S_1 = \{\he_1^w,\cdots, \he_{d-1}^w\}$ and $S_2 = \{\he_1^z,\cdots, \he_{d-1}^z\}$. See Figure~\ref{fig:ambiguous-graph} for a drawing of the graph when $d=3$. The intuition is to create a set of size $d-1$, $v_1,v_2,\cdots, v_{d-1}$ ($\{2,3\}$ in Figure~\ref{fig:ambiguous-graph}), that can be assigned to two possible hyperedges, both yielding a minimum preimage. 
\end{definition}

Now we prove that this graph satisfies the properties stated above.

\subsection{$\aG$ Satisfies Three Desired Properties}

\paragraph{Property 1: The Graph $\aG$ is Ambiguous.}
 This is shown in the following lemma.
\begin{lemma}\label{lem:bad-graph-non-unique-preimage}
 The graph $\aG$ from Defn.~\ref{def:aG} has two minimum preimages (so it is ambiguous).
\end{lemma}
\begin{proof}
Any preimage of $\aG$ must contain hyperedges in $S_1$ and $S_2$ as each $w_i^{(j)}$ ($z_i^{(j)}$) is only included in one clique. To include edges among $v_1,v_2,\cdots, v_{d-1}$, either $\he_1$ or $\he_2$ needs to be included in the preimage. Both $S_1\cup S_2\cup \{\he_1\}$ and $S_1\cup S_2\cup \{\he_2\}$ are valid preimages, so both are minimum preimages for $\aG$.
\end{proof}

\paragraph{Property 2: The Graph $\aG$ Appears with Probability $\Omega_n(1).$}
The next lemma shows that $\aG$ appears in $\pG$ with non-negligible probability using Lemma~\ref{lem:subgraph_threshold}.
\begin{lemma}\label{lem:bad-graph-contain-probability}
Let $\aG$ be as in Defn.~\ref{def:aG}. For any $ \delta\ge \frac{2d-4}{2d-1}$, 
    \[
    \Pr(\aG\subset \pG)=\Omega_n(1)\,.
    \]
\end{lemma}

\begin{proof}
Let us focus on one possible cover of $\aG$, $S_1\cup S_2\cup \{\he_1\}$,
\[
\Pr(\aG\subset \pG)\ge \Pr(S_1\cup S_2\cup \{\he_1\}\subset \hE_\rhG)\,.
\]
Recall in Lemma~\ref{lem:subgraph_threshold}, the probability of a hypergraph $\shG$ appear as a subgraph of $\rhG$ is described by  
\[
m(\shG) = \max_{\shG'\subset \shG}\frac{e_{\shG'}}{v_{\shG'}}\,.
\]
To use Lemma~\ref{lem:subgraph_threshold}, we need to calculate $m(S_1\cup S_2\cup \{\he_1\})$. The calculation is given in the following lemma.
\begin{lemma}\label{lem:density_bad_graph}
Let $S_1,S_2$, and $h_1$ be as in the Defn.~\ref{def:aG} of graph $G_{a,d}$ above. Then
    \[m(S_1\cup S_2\cup \{\he_1\}) = \frac{2d-1}{2d^2-5d+5}\,.\]
\end{lemma}
The calculation showing the lemma can be found in Appendix~\ref{sec:proof_density_bad_graph}.
Given Lemma~\ref{lem:density_bad_graph}, we have $\Pr(S_1\cup S_2\cup \{\he_1\}\subset \hE_\rhG)=\Omega_n(1)$ when 
\[
p=\Omega_n(n^{-\frac{2d^2-5d+5}{2d-1}})\,,
\]
i.e., when $\delta\ge \frac{2d-4}{2d-1}$. 
\end{proof}

\paragraph{Property 3: The Graph $\aG$ Forms a 2-connected Component with Probability $\Omega_n(1)$.}
What remains to be shown is that with probability $\Omega_n(1)$, not only does $\aG$ appears, but also $\cli(\aG)$ is a 2-connected component. In other words, we want to show that $\cli(\aG)$ has no 2-neighbors in $\cliG$, as shown in the following lemma.

\begin{lemma}\label{lem:branching}
For any $\hE_1\subset \binom{[n]}{d}$ with $|\hE_1|=O_n(1)$,
\[
    \Pr\b(\nei{\cliG}(\cli(\hE_1))\ne \emptyset | \hE_1\subset \hE_{\rhG}\b) =
    \begin{cases}
        O_n(n^{-(\frac{d-1}{d+1}-\delta)}) &\text{if }\delta<\frac{d-1}{d+1}\\
        1-\Omega_n(1) &\text{if }\delta=\frac{d-1}{d+1}
    \end{cases}
    \,.
\]
Recall that $\cli(\hE_1)=\cli(\proj(\hE_1))$.
\end{lemma}
The proof of the lemma can be found in Appendix~\ref{sec:brahcing}. The proof idea is similar to Lemma~\ref{lem:exp-dec}, where we consider all possible ways for a 2-neighbor to appear.

Since any preimage of $\aG$ has constant size, we can conclude that for any $\hE_a\subset \binom{[n]}{d} $ that is a preimage  of $\aG$, 
\[
\Pr(\nei{\cliG}(\cli(\hE_a))
= \emptyset|\hE_a\subset \hE_\rhG)=\Omega_n(1)\,.
\]
Combining this with Lemma~\ref{lem:bad-graph-contain-probability} yields the following Corollary.
\begin{corollary}\label{cor:bad-graph-probability}
    For any $\frac{2d-4}{2d-1}\le \delta\le \frac{d-1}{d+1}$, 
    \[
    \Pr\b(\cli(\aG)\text{ is a 2-connected component of }\cliG\b)=\Omega_n(1)\,.
    \]
\end{corollary}
Now we can prove Theorem~\ref{thm:upper}.

\subsection{Proof of Theorem~\ref{thm:upper}}
We can use Lemma~\ref{lem:bad-graph-in-isolation} by setting the graph $G_a$ to be $\aG$. By Lemma~\ref{lem:bad-graph-non-unique-preimage} and Corollary~\ref{cor:bad-graph-probability}, $\aG$ is graph with two minimum preimages and appears as a 2-connected component of $\cliG$ with constant probability. So for any algorithm $\cA$,
\[
\Pr(\cA(\pG)\ne \rhG)\ge \Omega_n(1)
\]
when $\frac{2d-4}{2d-1}\le \delta\le \frac{d-1}{d+1}$. By monotonicity stated in Lemma~\ref{lem:monotone}, we prove the theorem for $d=4$ and $5$.
And for $d=3$, this shows the impossibility when $2/5\le \delta\le 1/2$. The case of $d=3$, $\delta\ge 1/2$ is already shown in Theorem~\ref{thm:large-d}.
\qed

%% file: growth.tex
\section{Threshold of Growth for 2-Connected Components}\label{sec:growth-components}
In this section we prove that for sparse hypergraphs, all 2-connected components have constant size.
\component*

In order to prove the lemma, we formalize the notion of growing a subhypergraph.

\paragraph{Possible 2-neighbors of a sub-hypergraph and Definition of $\grow$.}
In Section~\ref{sec:const-size}, we discussed that the size of 2-connected components can be bounded by examining how 2-connected sub-hypergraphs can grow. Specifically, we will look at a sub-hypergraph $H\subset \rhG$ and the possible ways for $H$ to have a 2-neighbor. 

Suppose $\shG$ is a sub-hypergraph of $\rhG$ and $\cli(\proj(\shG))$ is 2-connected. 
Let $N(\shG)$ 
be the set of all possible 2-neighbors of $\cli(\proj(\shG))$. If $\cli(\proj(\shG))$ is a proper subset of a larger 2-connected hypergraph $\cli(\proj(\shG'))$ for some $\shG'\subset \rhG$, then there must exist $\he \in N(\shG)$ that is in $\cli(\proj(\shG'))$. An example is drawn in Figure~\ref{fig:possible-neighbors}, where $\shG$ only contains one hyperedge $\{1,2,3\}$, and all 3 possible 2-neighbors of $\{1,2,3\}$ forms $N(\shG)$.

\begin{figure}
\centering
\begin{minipage}{.45\textwidth}

\centering
\begin{tikzpicture}[scale=0.50]
    \coordinate (1) at (0,0);
    \coordinate (2) at (4,0);
    \coordinate (3) at (2,3.46);
    \coordinate (4) at (2,-2.46);
    \coordinate (5) at (-1,3.46);
    \coordinate (6) at (5,3.46);

    \fill[green!30] (1.center) -- (2.center) -- (3.center) -- cycle;
    \fill[cyan!30] (1.center) -- (2.center) -- (4.center) -- cycle;
    \fill[cyan!30] (1.center) -- (3.center) -- (5.center) -- cycle;
    \fill[cyan!30] (3.center) -- (2.center) -- (6.center) -- cycle;
    \node[draw, circle, fill=black, scale=0.5,label=left:1] (1) at (0,0) {};
    \node[draw, circle, fill=black, scale=0.5,label=right:2] (2) at (4,0) {};
    \node[draw, circle, fill=black, scale=0.5,label=above:3] (3) at (2,3.46) {}; 
    
    \node[draw, circle, fill=black, scale=0.5,label=below:4] (4) at (4) {}; 
    \node[draw, circle, fill=black, scale=0.5,label=left:5] (5) at (5) {}; 
    \node[draw, circle, fill=black, scale=0.5,label=right:6] (6) at (6) {}; 
    
    \node at (barycentric cs:1=1,2=1,3=1) {$\shG$};
    
    \draw (1) -- (2) -- (4) -- (1);
    \draw (2) -- (3) -- (6) -- (2);
    \draw (1) -- (3) -- (5) -- (1);
\end{tikzpicture}
\caption{An illustration of $N(\shG)$ when $d=3$. Here $\shG$ only has one hyperedge $\{1,2,3\}$, colored in green. $N(\shG)$ contains three possible 2-neighbors of $\cli(\proj(\shG))=\shG$, colored in blue.}
\label{fig:possible-neighbors}
\vspace{12mm}
\end{minipage}%
\hfill
\begin{minipage}{.45\textwidth}
\centering
\begin{tikzpicture}[scale=0.50]
    \coordinate (1) at (0,0);
    \coordinate (2) at (4,0);
    \coordinate (3) at (2,3.46);
    \coordinate (4) at (2,-2.46);
    \coordinate (7) at (-1,-2.46); 
    \coordinate (8) at (5,-2.46); 
    
    \fill[green!30] (1.center) -- (2.center) -- (3.center) -- cycle;
    \fill[yellow!30] (1.center) -- (2.center) -- (4.center) -- cycle; 
    \fill[orange!30] (1.center) -- (4.center) -- (7.center) -- cycle; 
    \fill[orange!30] (2.center) -- (4.center) -- (8.center) -- cycle; 
    
    \node[draw, circle, fill=black,scale=0.5, label=left:1] (1) at (1) {};
    \node[draw, circle, fill=black, scale=0.5,label=right:2] (2) at (2) {};
    \node[draw, circle, fill=black, scale=0.5,label=above:3] (3) at (3) {};
    \node[draw, circle, fill=black, scale=0.5,label=below:4] (4) at (4) {};
    \node[draw, circle, fill=black,scale=0.5, label=left:7] (7) at (7) {};
    \node[draw, circle, fill=black, scale=0.5, label=right:8] (8) at (8) {};
    
    \draw (1) -- (2) -- (4) -- (1);
    \draw (1) -- (4) -- (7) -- (1); 
    \draw (2) -- (4) -- (8) -- (2); 
    
    \node at (barycentric cs:1=1,2=1,4=1) {$h$};
    \node at (barycentric cs:1=1,2=1,3=1) {$\shG$};
    
    \draw[line width=2pt, red] (1) -- node[left] {$S_1$} (4);
    \draw[line width=2pt, red] (2) -- node[right] {$S_2$} (4);
    
    \draw (1) -- (2);
    \draw (1) -- (3);
    \draw (2) -- (3);
\end{tikzpicture}
\caption{An example of an element in $\grow(\shG,\he)$, consists of 3 hyperedges, $\{1,2,3\},\{1,4,7\}\text{ and }\{2,4,8\}$. Here $\shG$ contains one hyperedge $\{1,2,3\}$. $\he=\{1,2,4\}$. For $\he$ to be included in the 2-connected component, one way is to include $S_1^{\shG,\he} = \{1,4\}$ and $S_2^{\shG,\he} = \{2,4\}$ respectively in two hyperedges. }
\label{fig:example-grow}
\end{minipage}
\end{figure}

For $\he\in N(\shG)$ to appear in the 2-connected hypergraph $\cli(\proj(\shG'))$, every edge in $\proj(\he)\backslash \proj(\shG)$ should be covered in at least one hyperedge in $\shG'$.
Now let us examine the possible ways for this to happen.  
Let $\cS^{\shG,\he} \defeq \{S_1^{\shG,\he}, S_2^{\shG,\he}, \cdots, S_m^{\shG,\he}\}$ be the collection of subsets of $\he$ such that $\proj(S_i)\not\subseteq \proj(\shG)$ and $\proj(S_i)\ne \emptyset$. Let $m$ be the number of such subsets, and note that $m\le 2^d$. 
If a hyperedge covers an edge in $\proj(\he)\backslash \proj(\shG)$, it must intersect with $h$ at one of the sets in $\cS^{\shG,\he}$. For any $\he\in N(\shG)$, we define
\[
    \grow(\shG, \he) \defeq \{\shG\cup(\cup_{i\in I}\{\he_i\}): I\subseteq [m], \he_i\cap \he = S_i^{\shG,\he} , \proj(\cup_{i\in I}S_i^{\shG,\he}) = \proj(h)\backslash \proj(\shG)\}\,,
\]
and
\[
\grow(\shG) \defeq  \bigcup_{\he\in N(\shG)}\grow(\shG)\,.
\]
An example of an element in $ \grow(\shG, \he)$ is shown in Figure~\ref{fig:example-grow}.

Any 2-connected component can be achieved by ``growing'' multiple times from a single $d$-hyperedge.
\begin{lemma}\label{lem:grow-contain-all-components}
Assume w.l.o.g. that hyperedge $[d]$ is in a given $\shG$. We have that
\[
\{\shG:\text{$\cli(\proj(\shG))$ is 2-connected}\} = \bigcup_{t\ge 0}\grow^{(t)}([d])\,,
\]
where $\grow^{(t)}$ denotes applying $\grow$ $t$ times.
\end{lemma}
\begin{proof}
Choose an arbitrary size-$d$ hyperedge in $\shG$, without loss of generality assume it is $[d]$. We will grow it to $\shG$ by applying $\grow$ multiple times.

Specifically, we will show that given any $\shG_1\subset \shG$, there exists $\shG_2\in \grow(\shG_1)$ such that $\shG_1\subset \shG_2\subseteq \shG$. So $\shG$ can be obtained by growing $[d]$ finite number of times, as any hypergraph in $\grow(\shG_1)$ has more hyperedges than $\shG_1$.

If $\shG_1\subset \shG$, there must exists a $\he\in N(\{\shG_1\})$ in $\cli(\proj(\shG))$. As all edges in $\proj(\he)$ are in $\proj(\shG)$, we can select a set of hyperedges $\hE$ in $\shG$ that covers all edges in $\proj(\he)\backslash \proj(\shG_1)$. 
Let $\hE_i$ be the subset of hyperedges in $\hE$ that intersect with $\he$ at $S^{\shG_1,\he}_i$ 
\[
\hE_i\defeq \{\he'\in \hE: \he' \cap \he = S^{\shG_1,\he}_i\}\,.
\]
Let $\hE'$ be a set of edges by selecting one hyperedge from each non-empty $\hE_i$. We have $\proj(\hE')\supset \proj(\he)\backslash \proj(\shG_1)$.
Therefore, 
\[
\shG_2\defeq  \shG_1\cup\hE'\in  \grow(\shG_1,\he)\,.\qedhere
\]
\end{proof}

\paragraph{Decrease in the Expected Number of Appearances after Growth.}
To bound the probability of a large 2-connected component, we will show that any grow operation decreases the expected number of hypergraphs by a polynomial factor.

\begin{lemma}\label{lem:exp-dec}
Suppose $\delta<\frac{d-1}{d+1}$.
Let $X_\shG$ denote the number of appearance of $\shG$ in $\rhG$. Then for any $\shG$ with $O_n(1)$ number of vertices and any $\shG'\in \grow(\shG)$,
\[
\frac{\E X_{\shG'}}{\E X_{\shG}} = O_n\B(n^{-\b(\frac{d-1}{d+1}-\delta\b)}\B)\,.
\]
\end{lemma}

\begin{proof}
By Lemma~\ref{lem:number-appearance}, we have
\[
\E X_{\shG'} = \Theta_n(n^{v_{\shG'}}p^{e_{\shG'}}) \quad\text{and}\quad \E X_{\shG} = \Theta_n(n^{v_{\shG}}p^{e_{\shG}}) \,,
\]
So $\frac{\E X_{\shG'}}{\E X_{\shG}} = \Theta_n (n^{v_{\shG'}-v_{\shG}}p^{e_{\shG'}-e_{\shG'}})$. We bound this by considering all possible ways to grow $\shG$.

Suppose $\shG' \in \grow(\shG,\he)$ where $\he\in N(\shG)$. From the definition of $\grow(\shG,\he)$, let $\shG'\backslash\shG =\{\he_i\}_{i\in I}$ where
\[
\{\he_i\}_{i\in I} \text{ satisfy }   \he_i\cap \he = S_i^{\shG,\he} \text{ and } \proj(\cup_{i\in I}S_i^{\shG,\he}) = \proj(h)\backslash \proj(\shG)\,.
\] 
So $e_{\shG'}-e_{\shG'} = |I|$. The set of nodes in $\shG'$ but not in $\shG$ is given by the set of nodes that are in $\he$ but not $\shG$, which has size $d-k$, and the set of nodes that are in $\he_i$ but not in $\shG$ which has size $d-|S_i^{\shG,\he}|$. Therefore, we have
\[
v_{\shG'}-v_{\shG} \le  d-k+\sum_{i\in I} (d-|S_i^{\shG,\he}|)\,.
\]
We have inequality instead of equality because some vertices may be double-counted. 

Therefore,
\[
\begin{split}
\frac{\E X_{\shG'}}{\E X_{\shG}} &= O_n(n^{d-k+\sum_{i\in I} (d-|S_i^{\shG,\he}|)}p^{|I|})\,.
\end{split}
\]
Using $p=n^{-d+1+\delta}$, we have that for any $\shG'$,
\[
\frac{\E X_{\shG'}}{\E X_{\shG}} =O_n\B(n^{d-k}\cdot\max_{ \substack{I\subset[m]:\\(\proj(\he)\backslash \proj(\shG)) \subset \cup_i\proj(S_i^{\shG,\he})}} n^{-\sum_{i\in I}(|S_i^{\shG,\he}|-1-\delta)}\B)\,.
\]
Suppose $h$ shares $k$ nodes with $V(\shG)$. Here $k$ is at least 2 and at most $d$. 
We know $\proj(\he)\cap \proj(\shG)$ is a subset of a size-$k$ clique in $\he$. So the above expression can be further relaxed to $O_n(n^{d-k-g_k(\delta)})$ where
\begin{equation}\label{eq:gdelta}
g_k(\delta) \defeq \min_{\substack{I\subset [m]:\\ \b(\proj(\he)\backslash \binom{U_\he}{2}\b) \subset \cup_i\proj(S_i^\he)}} \sum_{i\in I}(|S_i^\he|-1-\delta)\,.
\end{equation}
Here $U_\he$ is a size-$k$ subset of $\he$. Here $g_k(\delta)$ does not depend on the choice of $\he$. We will show a a bound on $\min_k \{g_k(\delta)+k-d\}$ in Lemma~\ref{lem:cover-bound} at the end of the section. Given the bound in Lemma~\ref{lem:cover-bound}, we have for any $\delta<\frac{d-1}{d+1}$,
\[
\frac{\E X_{\shG'}}{\E X_{\shG}} = O_n\B(n^{-\b(\frac{d-1}{d+1}-\delta\b)}\B)\,.\qedhere
\]
\end{proof}

\paragraph{Bound on Component Size via Number of Growth Steps.}
Now we are ready to bound the size of 2-connected components.

\begin{proof}[Proof of Lemma~\ref{lem:component-constant-size}]
Let $t= 2(\frac{d-1}{d+1}-\delta)\inv$. 
We will show that with high probability any hypergraph in $\grow^{(t)}([d])$ does not appear in $\rhG$.
First, $X_{[d]} = \binom{n}{d}p = \Theta_n(n^{1+\delta})$. By Lemma~\ref{lem:exp-dec}, we have for any graph $\shG$ in $\grow^{(t)}([d])$,
\[
X_{\shG} =O_n\B(n^{1+\delta}\cdot n^{-t\b(\frac{d-1}{d+1}-\delta\b)}\B) = O_n(n^{\delta-1}) = o_n(1)\,.
\]
By Markov's inequality, $\Pr(\shG\subset \rhG) = o_n(1)$. There are $O_n(1)$ graphs in $\grow^{(t)}([d])$, so by the union bound,
\[
\Pr\b(\shG \subset \rhG\text{ for some } \shG \in\grow^{(t)}([d])\b) = o_n(1)\,.
\]
Note that for any $t'>t$, hypergraphs in $\grow^{(t')}([d])$ contains one of the hypergraphs in $\grow^{(t)}([d])$ and therefore do not appear with high probability.
So by Lemma~\ref{lem:grow-contain-all-components} with high probability, all 2-connected components in $\rhG$ are in 
\[
\bigcup_{i=0}^t \grow^{(t)}([d])\,.
\]
Since the grow operation increases the number of hyperedges by at most $2^d$, with high probability all 2-connected components in $\rhG$ have size at most
\[
1+t2^d
= 1+\frac{2^{d+1}}{\frac{d-1}{d+1}-\delta}\,,
\]
as stated in the lemma. 
\end{proof}

The lemma below shows a bound on $g_k(\delta)$ in \eqref{eq:gdelta} that was used in the proof of Lemma~\ref{lem:exp-dec}.
\begin{lemma}\label{lem:cover-bound}
For any $\delta\le \frac{d-1}{d+1}$,
    \[
    \min_{2\le k\le d,k\in \mathbb{Z}} \{g_k(\delta)+k-d\}\ge \frac{d-1}{d+1}-\delta\,.
    \]
\end{lemma}
\begin{proof}
Recall that \[
g_k(\delta) \defeq \min_{\substack{I\subset [m]:\\ \b(\proj(\he)\backslash \binom{U_\he}{2}\b) \subset \cup_i\proj(S_i^\he)}} \sum_{i\in I}(|S_i^\he|-1-\delta)\,.
\]
The function takes the minimum of linear functions of $\delta$, so this is a piece-wise linear function. Since every linear function has slope at most $-1$, $g_k(\delta)$ also has slope at most $-1$ in each piece.

Next we will lower bound the value of $g_k(\delta)$ by a series of relaxations on the domain of the optimization.
For any set of $\{S_i^\he\}_{i\in I}$, the union of edges is a superset of all edges in $\proj(\he)\backslash \binom{U_\he}{2}$. So 
we can get a lower bound on $g_k(\delta)$ by relaxing the set of possible $I$ to be the set of cliques with at least $\binom{d}{2}-\binom{k}{2}$  number of edges. Also each clique in the set $I$ that reaches minimum contains a unique edge, so $|I|\le \binom{d}{2}-1$.
\[
\min_{\substack{I\subset [m],|I|\le \binom{d}{2}-1:\\ \sum_{i\in I}\binom{|S_i^\he|}{2} \ge \binom{d}{2}-\binom{k}{2}}} \sum_{i\in I}(|S_i^\he|-1-\delta)
\,.
\]
Substituting $|S_i^\he|$ with $x_i$ and relaxing it to real numbers, we get another lower bound on $g_k(\delta)$:
\[
\min_{M\in \mathbb{Z}^+,M\le \binom{d}{2}-1}
\min_{\substack{x_1,x_2,\cdots, x_M\ge 2:\\ \sum_{i=1}^M\frac{x_i(x_i-1)}{2} \ge \binom{d}{2}-\binom{k}{2}} } \sum_{i=1}^M(x_i-1-\delta)
\,.
\]
So 
\[
\begin{split}
&\min_k \{g_k(\delta)+k-d\}\\
&\ge 
\min_{\substack{M\in \mathbb{Z}^+,\\ M\le \binom{d}{2}-1}}
\min_{2\le k\le d}
\min_{\substack{x_1,x_2,\cdots, x_M\ge 2:\\ \sum_{i=1}^M\frac{x_i(x_i-1)}{2} \ge \binom{d}{2}-\binom{k}{2}} } \B\{\sum_{i=1}^M(x_i-1-\delta)+k-d\B\}\\
& = 
\min_{\substack{M\in \mathbb{Z}^+,\\ M\le \binom{d}{2}-1}}
\min_{\substack{x_0,x_1,\cdots, x_M\ge 2:\\ \sum_{i=0}^M\frac{x_i(x_i-1)}{2} \ge \binom{d}{2}} }\B\{ \sum_{i=0}^M x_i - M(1+\delta)-d\B\}\,.
\end{split}
\]
Here in the equality we substituted $k$ with $x_0$.
By setting $y_i=\frac{x_i(x_i-1)}{2}$, the above can be written as 
\[
\min_{\substack{M\in \mathbb{Z}^+,\\ M\le \binom{d}{2}-1}}
\min_{\substack{y_0,y_1,\cdots, y_M\ge 1:\\ \sum_{i=0}^M y_i \ge \binom{d}{2}} }\B\{ \sum_{i=0}^M \B(\frac{1+\sqrt{1+8y_i}}{2}\B) - M(1+\delta)-d\B\}\,.
\]
For a fixed $M$, this is minimizing a concave function of $y$ over a polyhedron. So the minimum is either at a vertex or infinity. The latter is obviously not the minimum. So the minimum is at a vertex of the following polyhedron:
\[
P\defeq \B\{y:y_i\ge 1, \sum_{i=0}^M y_i \ge \binom{d}{2}\B\}\,.
\]
By symmetry of the function and $P$ under permutation of coordinates, we can consider one of the vertices without loss of generality. Let $y_0=y_1=\cdots=y_{M-1} = 1$, $y_M=\binom{d}{2}-M$, we have that the above is equal to
\[
\min_{\substack{M\in \mathbb{Z}^+,\\ M\le \binom{d}{2}-1}} \B\{ 2M-M(1+\delta)-d+  \frac{1+\sqrt{1+8(d(d-1)/2-M)}}{2}\B\}
\]
The function is concave in $M$, so the minimum is at $M=1$ or $M=\binom{d}{2}-1$. When $\delta=\frac{d-1}{d+1}$ and $M=\binom{d}{2}-1$, the function is 0. When $\delta=\frac{d-1}{d+1}$ and $M=1$,  the function is 
\[
\frac{1+\sqrt{1+8(d(d-1)/2-1)}}{2} -d +\frac{2}{d+1}
\]
One can solve that the roots of this function are $\frac{1}{2}\pm \frac{\sqrt{17}}{2}$ and when $d\ge3>\frac{1}{2}+ \frac{\sqrt{17}}{2}$, this function is always positive.
Therefore, when $\delta=\frac{d-1}{d+1}$, $\min_k \{g_k(\delta)+k-d\} \ge 0$. Since this is a piece-wise function of $\delta$ with slope at most $-1$, we know for any $\delta\le \frac{d-1}{d+1}$, $\min_k \{g_k(\delta)+k-d\} \ge \frac{d-1}{d+1} -\delta$.
\end{proof}

%% file: FutureDirections.tex
\section{Open Problems}
We list a number of open problems:
\begin{enumerate}
    \item Prove or disprove the conjecture that the threshold $\cdel d$ for $d=4,5$ is at $\frac{2d-4}{2d-1}$.
    \item Instead of having uniform hypergraphs, we can consider general hypergraphs with hyperedges of different sizes. If hyperedges of different sizes appear independently at random, when is exact recovery possible?
    \item Instead of exact recovery, we might aim for only \emph{almost} exact recovery or \emph{partial} recovery, where we want to recover $1-o_n(1)$ or $\Omega_n(1)$ fraction of hyperedges instead of all hyperedges. Where is the optimal threshold for almost exact recovery.
    \item What is the reconstruction threshold for recovery from the similarity matrix rather than the projected graph? In other words, how much does it help to know the number of hyperedges an edge belongs to? 
    \item A problem exhibits a statistical-computational gap if the information theoretically optimal performance cannot be reached by an efficient algorithm. For $d=3,4,5$, we have shown that the information theoretically optimal algorithm is always efficient whenever exact recovery is possible. When $d\ge 6$ what is the threshold $\cdel d$? If the threshold is above $\frac{d-1}{d+1}$, 
    is there a statistical-computational gap? In other words, is it still possible to efficiently achieve exact recovery for any $\delta\le \cdel d$?
    \item If we consider the random hypergraph as a random bit string of length $\binom{n}{d}$, the projected graph as a bit string of length $\binom{n}{2}$. The operator $\proj$ can be viewed as an information channel. How much entropy is lost with the projection operation $\proj$?
    \item How well does the greedy algorithm (Algorithm~\ref{alg:greedy}) work? What is the threshold for the greedy algorithm to exactly recover the original hypergraph?
    \item The running time of Algorithm~\ref{alg:map} depends exponentially on $\b(\frac{d-1}{d+1}-\delta\b)\inv$. Is this dependence necessary?
\end{enumerate}

%% file: Lowerbound.tex
\section{Computer-Assisted Proof of Small Subgraph Preimage Uniqueness}\label{sec:algo-proof}


Recall that as discussed in Section~\ref{sec:main-idea-reconstruction}, in order to prove Theorem~\ref{thm:delta-lower}, all we need to do is check the non-existance of ambiguous graphs. Specifically, we need to prove the following.
\lowerbound*

In this section, we provide a proof for the claim, with computer assistance.
\paragraph{Depth First Search (DFS) over Hypergraphs.}
First, instead of searching over graphs $G_a$, we can search over preimages of the graphs, i.e., hypergraphs. The claim in Lemma~\ref{lem:delta-lowerbound} is equivalent to 
\begin{gather*}
    \text{For any sub-hypergraph $ \shG$ where $\cli(\proj(\shG))$ is 2-connected and  $\Pr(\shG\subset \rhG)=\Omega_n(1)$,}\\
\text{$\shG$ has unique minimum preimage.}
\end{gather*}
We will prove a sufficient condition for the claim to hold by replacing $\Pr(\shG\subset \rhG)=\Omega_n(1)$ with $\E X_{\shG}=\Omega_n(1)$.
Therefore, we want to search over all hypergraphs $ \shG$ where $\cli(\proj(\shG))$ is 2-connected and  $\E X_{\shG}=\Omega_n(1)$. If all such graphs have unique minimum preimage, then Lemma~\ref{lem:delta-lowerbound} is proven.

Lemma~\ref{lem:grow-contain-all-components} implies that it suffices to consider $\grow^{(t)}([d])$ for $t=1,2,\ldots$ in order to capture all $\shG$ where $\cli(\proj(\shG))$ is 2-connected. 
Based on $\grow$, we define a depth-first-search tree $T$ where nodes of the tree are hypergraphs from $\cup_{t\ge 1}\grow^{(t)}([d])$. The root of $T$ is a hypergraph with a single $d$-hyperedge $[d]$. For any node $\shG$ on the tree, its children is all hypergraphs in $\grow(\shG)$, as shown in Figure~\ref{fig:DFS}.
We therefore start with a hypergraph with a single hyperedge and perform a depth first search over the tree.

The depth of the search is also bounded by a $2/(\frac{d-1}{d+1}-\delta)=O_n(1)$, as Lemma~\ref{lem:exp-dec} tells us that each growth step decreases the expected number of appearances by a polynomial factor. 

\begin{figure}
    \centering
\begin{tikzpicture}[every node/.style={draw,circle}, 
                    level distance=2cm,
                    level 1/.style={sibling distance=6cm},
                    level 2/.style={sibling distance=3cm}]
    \node (G1) {$\shG_1$}
    child { node (G2) {$\shG_2$}
        child { node {$\shG_4$} }
        child { node {$\shG_5$} }
    }
    child { node (G3) {$\shG_3$}
        child { node {$\shG_6$} }
        child { node {$\shG_7$} }
    };

    \begin{scope}[on background layer]
    \node[ellipse, draw=none, fit=(G2) (G3) , inner sep=1mm,fill=violet!30] {$\grow(\shG_1)$};
    \end{scope}

\end{tikzpicture}
    \caption{Depth First Search Tree over Graphs. The children of any hypergraph $\shG$ is the set $\grow(\shG)$.}
    \label{fig:DFS}
\end{figure}
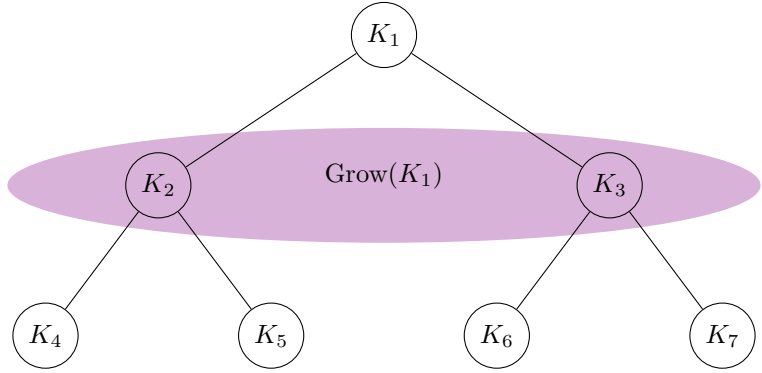

\paragraph{Pruning by Bounding the Expected Number of Appearances.}
The benefit of using the structure of $\grow$ to search instead of an arbitrary order is that we can do the following pruning.
By Lemma~\ref{lem:exp-dec}, during the depth first search, children always have a smaller expected number of appearances than the parent.
So if we reach a hypergraph $\shG$ with $o_n(1)$ expected number of appearance, we can stop the search on this branch in the depth first search tree, as any children of the graph will also have $o_n(1)$ expected number of appearances. 


\paragraph{Improve the Root of the DFS Tree.}
We use the lemma below to further narrow down the search.
Given Lemma~\ref{lem:ambiguous-graph-has-connected-preimage}, instead of searching from $[d]$, we can start the search from \emph{two} hyperedges that overlap on at least 2 vertices. This turns out to dramatically decrease the depth of the search. 

\begin{lemma}\label{lem:ambiguous-graph-has-connected-preimage}
    Fix $\delta<\frac{2d-4}{2d-1}$. If $G_a$ is an ambiguous graph, then either
    \begin{itemize}
        \item $\Pr(G_a\subset \pG)=o_n(1)$ or 
        \item one of the minimum preimages of $G_a$ contains two hyperedges that share two vertices.
    \end{itemize}
\end{lemma}
\begin{proof}
Let $G_a$ be an ambiguous graph and for every minimum preimage of $G_a$, any two hyperedges in the graph share at most one vertex. 

Let $\shG_1$ and $\shG_2$ be two minimum covers of $G_a$. Let $\hE_c = \hE_{\shG_1}\cap \hE_{\shG_2}$, $\hE_1 = \hE_{\shG_1}\backslash \hE_{\shG_2}$ and $\hE_2 = \hE_{\shG_2}\backslash \hE_{\shG_1}$. Then we have $\hE_{\shG_1} $ is partitioned to $\hE_c$ and $\hE_1$, $\hE_{\shG_2} $ is partitioned to $\hE_c$ and $\hE_2$. $\hE_1\cap \hE_2=\emptyset$.
Since any two hyperedges in the graph share at most one vertex, $\proj(\hE_c)\cap \proj(\hE_1) = \emptyset$. So
\[
\proj(\hE_1) = G_a\backslash \proj(\hE_c) = \proj(\hE_2)\,.
\]
We will show that $\proj(\hE_1)$ appears in $\pG$ with $o_n(1)$ probability, then by Lemma~\ref{lem:exp-dec}, $G_a$ also appears in $\pG$ with $o_n(1)$ probability.

Suppose $|\hE_1| = k$, then so is $|\hE_2|$, as $\shG_1$ and $\shG_2 $ have the same size. The total degree of $\hE_1$ is therefore $dk$. For any node $v\in V(\hE_1)$, $v$ is in one of the hyperedges $\he$ in $\hE_2$. There are $d-1$ edges in $\proj(\hE_1)$ between $v$ and other nodes in $\he$. All $d-1$ edges are included in some hyperedges in $\hE_1$. But they cannot be in a single hyperedge in $\hE_1$, otherwise that hyperedge would be $\he$, contradicting with $\hE_1\cap \hE_2=\emptyset$. So $v$ has degree at least 2 in $\hE_1$. Therefore, 
\[
v_{\hE_1} \le \frac{dk}{\text{minimum degree}} \le dk/2\,.
\]
So by Lemma~\ref{lem:number-appearance}, 
\[
\Pr(\hE_1\subset \rhG) \le \E X_{\hE_1} = \Theta_n(n^{v_{\hE_1}}p^{e_{\hE_1}}) = O_n(n^{dk/2}p^k) = O_n(n^{-d/2+1+\delta})\,.
\]
Using $\delta<\frac{2d-4}{2d-1}$, the above is $O_n(n^{2-d/2-3/(2d-1)})$. When $d=3$, this is $O_n(n^{-1/10}) = o_n(1)$. When $d\ge 4$, this is $O_n(n^{-3/(2d-1)}) = o_n(1)$.
\end{proof}

\paragraph{The Search Algorithm}
The algorithm is given as Algorithm~\ref{alg:search}.
Lines 14 to 16 enumerates two hyperedges that overlap on at least 2 vertices and starts the DFS search from this hypergraph. The first input of the procedure is the hypergraph itself, the second input of the procedure is the expected number of the hypergraphs in $\rhG$. For two hyperedges that overlap on $k$ vertices, the expectation is $\Theta_n(n^{2d-k}p^{2})$, by Lemma~\ref{lem:number-appearance}.

In the procedure DFS, Lines 3 to 8, we check whether the graph is ambiguous. Lines 9 and 10 examine the expected number of the current graph $\shG$ in $\rhG$. If it is $o_n(1)$, this branch of the search can be pruned as any graph growing from $\shG$ would have vanishing probability of appearing. Line 11 to Line 13 continues to search from all graphs in $\grow(H)$.
\paragraph{Results of Algorithm~\ref{alg:search}.}
So far we have shown that Algorithm~\ref{alg:search} will find all  hypergraphs $\shG$ such that $\E X_{\shG}=\Omega_n(1)$ and $\proj(\shG)$ is ambiguous. What remains is running the code. 

When $d=3$ and $\delta=2/5$, the algorithm only finds one ambiguous graph, $\aG$. For $d=4,5$ and $\delta=1/2$, the algorithm did not find any ambiguous graph. Thus we proved Lemma~\ref{lem:delta-lowerbound}.

Unfortunately despite our efforts to optimize the search algorithm, it still takes a rather long time when $d\ge 4$ and $\delta>1/2$. We leave it as an open problem to determine the correct threshold $\delta^*$ when $d=4,5$. We conjecture that the threshold matches the upper bound $\frac{2d-4}{2d-1}$.

\begin{algorithm}
\caption{Search Algorithm}\label{alg:search}
\begin{algorithmic}[1] 
\State Input: $d,p=n^{-d+1+\delta}$.


\Procedure{DFS}{a hypergraph, $\shG$; and the expected number of $\shG$ in $\rhG$, $U$}

\For{$k=1,2,\cdots$}
\State Find all preimages of $\proj(\shG)$ with $k$ hyperedges.
\If{One such preimage is found}
\State \textbf{continue}
\EndIf
\If{Two such preimage is found}
\State \textbf{output:} An ambiguous graph $\proj(\shG)$ is found.
\EndIf
\EndFor

\If{$U=O_n(1)$}
\State \textbf{return}
\EndIf
\For{$\he\in N(\shG)$}
\For{$\shG'\in \grow(\shG,\he)$}
\State DFS($\shG'$,$ U\cdot n^{v_{\shG'}-v_{\shG}}p^{e_{\shG'}-e_{\shG}}$)
\EndFor
\EndFor
\EndProcedure
\For{$k=2,\cdots,d-1$} 
    \State Let $\shG_k$ be the graph with two hyperedges that share $k$ vertices.
    \State DFS($\shG_k$, $n^{2d-k}p^2$).
\EndFor
\end{algorithmic}
\end{algorithm}

%% file: ProofNaiveAlgo.tex
\section{Analysis of the Maximum Clique Cover Algorithm}\label{sec:naive}
In this section we will prove Theorem~\ref{thm:naive}.
Specifically, we want to show that the event,
\[
\text{a size-$d$ clique $\he$ appear in $\pG$ but is not a hyperedge in $\rhG$},
\]
happens with probability $o_n(1)$ when $\delta<\frac{d-3}{d}$ but happens with at least constant probability when $\delta\ge \frac{d-3}{d}$.

Let us prove the two claims respectively in the following two lemmas. 

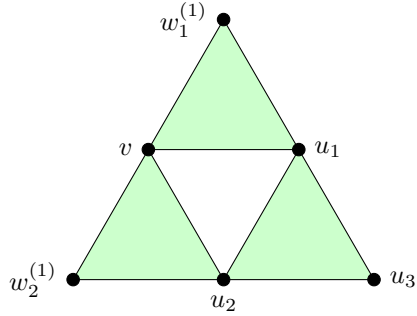
\begin{figure}[ht]
\centering
\begin{tikzpicture}
\coordinate (A) at (0,0);
\coordinate (B) at (2,0);
\coordinate (C) at (1,1.73);
\coordinate (D) at (-1,-1.73);
\coordinate (E) at (1,-1.73);
\coordinate (F) at (3,-1.73);  
\fill[fill=green!20] (A)--(B)--(C);
\fill[fill=green!20] (B)--(E)--(F);
\fill[fill=green!20] (A)--(D)--(E);

\node[draw, circle, scale=0.5, fill=black, label=left:$v$] at (A) {};
\node[draw, circle, fill=black, scale=0.5, label=right:$u_1$] at (B) {};
\node[draw, circle, fill=black,  scale=0.5,label=left:$w_1^{(1)}$] at (C) {};
\node[draw, circle, fill=black,  scale=0.5,label=left:$w_2^{(1)}$] at (D) {};
\node[draw, circle, fill=black,  scale=0.5,label=below:$u_2$] at (E) {};
\node[draw, circle, fill=black, scale=0.5,label=right:$u_3$] at (F) {};

\draw (A)--(B)--(C)--cycle;
\draw (A)--(D)--(E)--cycle;
\draw (B)--(E)--(F)--cycle;
\end{tikzpicture}
\caption{Illustration of $\naH$ when $d=3$.}
\label{fig:fake-hyperedge}
\end{figure}

\begin{lemma}
    When $\delta\ge \frac{d-3}{d}$, 
    \[
    \Pr\B(\exists \he\in\binom{[n]}{d} \text{ s.t. } \he  \text{ is a clique in $\pG$ but } \he\notin \hE_\rhG\B)=\Omega_n(1)\,.
    \]
\end{lemma}
\begin{proof}
We will focus on a sufficient condition that $\he = \{v,u_1,\cdots,u_{d-1}\}$ is a clique but does not appear in  $\hE_\rhG$. Then we will use second moment method to lower bound the probability of that sufficient condition happening at one of the hyperedges. The sufficient condition is that $\he\notin \hE_\rhG$ and all of the following hyperedges are in $\hE_\rhG$:
\begin{itemize}
    \item $\{u_1,u_2,\cdots,u_d\}$ and
    \item $\{v,u_i,w_i^{(1)},\cdots, w_i^{(d-2)}\}$ for all $1\le i\le d-1$.
\end{itemize}
Let $\naH$ denote the hypergraph, see Figure~\ref{fig:fake-hyperedge} for an illustration (we abuse notation and let $\naH$ include the non-hyperedge $\he$).
There are $v_\naH=d^2-2d+3$ nodes  and $e_\naH=d$ edges in $\naH$. $p=n^{-d+1+\delta} =\Omega_n (n^{-d+2+3/d}) = \Omega_n (n^{-v_\naH/e_\naH})$.
Let $\naH_1,\naH_2,\cdots,\naH_t$ be all copies of such sub-hypergraph on the complete graph of $[n]$, we have
\[
t=\binom{n}{v_\naH}\frac{(v_\naH)!}{\aut(\naH)}=\Theta_n(n^{v_\naH})\,.
\]
Here $\aut(\naH)$ is the number of automorphisms of $\naH$.
Let $I_i$ be the indicator that $\naH_i$ is in $\rhG$. And $X_\naH = \sum_{i=1}^tI_i$ be the number of such event happening. We have
\[
\E [X_\naH] = tp^{e_\naH}(1-p) = \Theta_n(n^{v_\naH}p^{e_\naH})\,.
\]
And 
\[
\var(X_\naH) = \sum_{i=1}^t\sum_{j=1}^t \cov(I_iI_j) = 
\sum_{i=1}^t\sum_{j=1}^t (\Pr(I_i=I_j=1)-\Pr(I_i=1)\Pr(I_j=1))\,.
\]
We have $\Pr(I_i=1) = \Pr(I_j=1) = p^{d}(1-p)$. If the non-hyperedge in $\naH_i$ overlaps with a hyperedge in $\naH_j$ or vice versa, then $\cov(I_i,I_j)$ are negative. So to get an upper bound of $\var(X_\naH)$, we only consider pairs $(\naH_i,\naH_j)$ that are positively correlated. 
Consider pairs $(\naH_i,\naH_j)$ such that $\naH_i\cap \naH_j = K$, where $K\subset \naH$ is a sub-hypergraph of $\naH$ with non-empty edge set.
\begin{align*}
\var(X_\naH) &= O_n\B( \sum_{\substack{K \subseteq \naH,\\ e_K > 0}} n^{2v_\naH - v_K} \left( p^{2e_\naH - e_K} - p^{2e_\naH} \right) \B) \\
&= O_n\B( n^{2v_\naH} p^{2e_\naH} \sum_{\substack{K \subseteq H,\\ e_K > 0}} n^{-v_K} p^{-e_K} \B).
\end{align*}
Then from Lemma~\ref{lem:non-zero}, 
\[
\Pr(X_\naH\ne 0) \ge \frac{(\E X_\naH)^2}{(\E X_\naH)^2+\var(X_\naH)} = \frac{1}{1+O_n(\sum_{\substack{K\subseteq \naH,\\ e_K > 0}} n^{-v_K} p^{-e_K})}\,.
\]
We can easily check that for any $K\subset \naH$, $\frac{e_K}{v_K}\le \frac{e_\naH}{v_\naH}$. So by $p=\Omega_n (n^{-v_\naH/e_\naH})$, we have
\[
\sum_{\substack{K\subseteq \naH,\\ e_K > 0}} n^{-v_K} p^{-e_K} = \sum_{\substack{K\subseteq \naH,\\ e_K > 0}} O_n(n^{-v_\naH} n^{v_\naH }) = O_n(1)\,.
\]
This means $\Pr(X_\naH\ne 0) = \Omega_n(1)$.
\end{proof}

\begin{lemma}
When $\delta<\frac{d-3}{d}$,
\[
    \Pr\B(\exists \he\in\binom{[n]}{d} \text{ s.t. } \he  \text{ is a clique in $\pG$ but } \he\notin \hE_\rhG\B)=o_n(1)\,.
\]
\end{lemma}
By union bound over all possible cliques in $\binom{[n]}{d}$, the above probability is upperbounded by 
\[
\binom{n}{d}\Pr([d] \text{ is a clique in $\pG$ but } [d]\notin \hE_\rhG) = \binom{n}{d}(1-p)\Pr(\forall i,j\in [d],\exists \he\in \hE_\rhG, i,j\in \he\b|[d]\not\in \hE_\rhG)
\]

Now we split the event as follows. Let $\mathcal{S} = \{S_1,S_2,\cdots, S_{m}\}$ be all non-empty proper subsets of $[d]$ with size at least 2, $m=2^d-2-d$. Let $A_i$ be the event that at least one of edges in the set 
\[\hE_i=
\B\{\he\in \binom{[n]}{d}\B|\he\cap [d] = S_i\B\}
\]
is included in $\hE_\rhG$. Then the event that $[d]$ is a clique in $\pG$ but $[d]\notin \hE_\rhG$ is equivalent to the event that every pair $j,k\in[d]$ is included in some $S_i$ where $A_i$ happens. We have
\begin{align*}
&\  \Pr(\forall i,j\in [d],\exists \he\in \hE_\rhG, i,j\in \he\b|[d]\not\in \hE_\rhG) \\
&= \sum_{I\subset [m]} \1\{\forall j,k\in[r], \exists i\in I, j,k\in S_i\} \Pr((\cap_{i\in I}A_i)\cap(\cap_{i\in [m]\backslash [I]}A_i^c) )\\
&\le \sum_{I\subset [m]} \1\{\forall j,k\in[r], \exists i\in I, j,k\in S_i\} \Pr(\cap_{i\in I}A_i)\\
&= \sum_{I\subset [m]} \1\{\forall j,k\in[r], \exists i\in I, j,k\in S_i\} \prod_{i\in I}\Pr(A_i) \numberthis\label{eq:union-events}
.
\end{align*}
The last inequality is because $\hE_i$ are disjoint sets of edges, so $A_i$ are independent events. 

Note that there are $\binom{n}{d-|S_i|}$ edges in $\hE_i$, we have 
\[
\Pr(A_i)\le  1-(1-p)^{\binom{n}{d-|S_i|}}\le pn^{d-|S_i|}\le n^{-(|S_i|-1-\delta)}
.
\]
Since there are $O_n(1)$ terms in \eqref{eq:union-events}, we have
\[
\Pr(\forall i,j\in [d],\exists \he\in \hE_\rhG, i,j\in \he\b|[d]\not\in \hE_\rhG) = O_n\b(n^{-g_0(\delta)}\b).
\]
Recall that \[
g_0(\delta) = \min_{\substack{I\subset [m]:\\ \proj([d]) \subset \cup_i\proj(S_i)}} \sum_{i\in I}(|S_i|-1-\delta)\,.
\]
Therefore, 
\begin{equation*}
 \Pr\B(\exists \he\in\binom{[n]}{d} \text{ s.t. } \he  \text{ is a clique in $\pG$ but } \he\notin \hE_\rhG\B)  = O_n\b(n^{d-g_0(\delta)}\b)  \,.
\end{equation*}
The calculation of $g_0(\delta)$ is a nontrivial combinatorial optimization problem. 

Assuming Lemma~\ref{lem:optimal-clique-cover}, and note that $g_0(\delta)$ is strictly decreasing in $\delta$, we have for any $\delta<\frac{d-3}{d}$, $g_0(\delta)<d$. And thus 
\begin{equation*}
 \Pr\B(\exists \he\in\binom{[n]}{d} \text{ s.t. } \he  \text{ is a clique in $\pG$ but } \he\notin \hE_\rhG\B)  = O_n\b(n^{d-g_0(\delta)}\b)=o_n(1)  \,.
\end{equation*}

\begin{lemma}\label{lem:optimal-clique-cover}
Let $\cS = \{S_1,S_2,\cdots S_m\}$ be the set of all proper subsets of $[d]$ with size at least 2. When $\delta = \frac{d-3}{d}$, 
\[
g_0(\delta) =\min_{\substack{\cS'\subset \cS:\\ \proj([d]) \subset \proj(\cS')}} \sum_{S\in \cS'}(|S|-1-\delta)=d\,,
\]
achieved by the following set of subsets:
$\{2,3,\cdots,d\},\{1,2\},\{1,3\},\cdots,\{1,d\}$.
\end{lemma}

\begin{proof}
We will prove this under two separate cases. The first case is when $|\cS'|\ge d$, this is done by relaxation of the problem to real number. The second case is when $|\cS'|\le d$, which is done by induction on $d$.

\paragraph{Case 1: $|\cS'|\ge d$.}Let's begin with the first case. We will prove 
\[
f(\delta) =\min_{\substack{\cS'\subset \cS,|\cS'|\ge d:\\ \proj([d]) \subset \proj(\cS')}} \sum_{S\in \cS'}(|S|-1-\delta)=d\,.
\]
This part of the proof is similar to what we did in Lemma~\ref{lem:cover-bound}.
We can get a lower bound on $f(\delta)$ by relaxing the set of possible $\cS'$ to be the set of cliques with at least $\binom{d}{2}$  number of edges. Also each clique in the set $\cS'$ that reaches minimum contains a unique edge, so $|\cS'|\le \binom{d}{2}$. We have
\[
f(\delta)\ge 
\min_{\substack{\cS'\subset \cS,d\le |\cS'|\le \binom{d}{2}:\\ \sum_{i\in I}\binom{|S_i|}{2} \ge \binom{d}{2} }} \sum_{i\in I}(|S_i|-1-\delta)
\,.
\]
Substituting $|S_i|$ with $x_i$ and relaxing it to real numbers, we get another lower bound on $f(\delta)$:
\[
\min_{M\in \mathbb{Z}^+,d\le M\le \binom{d}{2}}
\min_{\substack{x_1,x_2,\cdots, x_M\ge 2:\\ \sum_{i=1}^M\frac{x_i(x_i-1)}{2} \ge \binom{d}{2}} } \sum_{i=1}^M(x_i-1-\delta)
\,.
\]
By setting $y_i=\frac{x_i(x_i-1)}{2}$, the above can be written as 
\[
\min_{\substack{M\in \mathbb{Z}^+,\\d\le M\le \binom{d}{2}}}
\min_{\substack{y_1,\cdots, y_M\ge 1:\\ \sum_{i=0}^M y_i \ge \binom{d}{2}} }\B\{ \sum_{i=1}^M \B(\frac{1+\sqrt{1+8y_i}}{2}\B) - M(1+\delta)\B\}\,.
\]
For a fixed $M$, this is minimizing a concave function of $y$ over a polyhedron. So the minimum is either at a vertex or infinity. The later is obviously not the minimum. So the minimum is at a vertex of the following polyhedron:
\[
P= \B\{y:y_i\ge 1, \sum_{i=0}^M y_i \ge \binom{d}{2}\B\}\,.
\]
By symmetry of the function and $P$ under permutation of coordinates, we can consider one of the vertices without loss of generality. Let $y_1=y_2=\cdots=y_{M-1} = 1$, $y_M=\binom{d}{2}-M+1$, we have that the above is equal to
\[
\min_{\substack{M\in \mathbb{Z}^+,\\ 2\le M\le \binom{d}{2}}} \B\{ (M-1)(1-\delta)-1-\delta+  \frac{1+\sqrt{1+8(d(d-1)/2-M+1))}}{2}\B\}\,.
\]
The function is concave in $M$, so the minimum is at $M=2$ or $M=\binom{d}{2}$. When $\delta=\frac{d-3}{d}$ and $M=\binom{d}{2}$, the function is $3(d-1)/2>d$. When $\delta=\frac{d-3}{d}$ and $M=d$,  the function is $d$. So $f(\delta)\ge d$ when $\delta =\frac{d-3}{d} $
\paragraph{Case 2: $|\cS'|\le d$.} Next we prove the second case. We will show 
\[
h(d) \defeq\min_{\substack{\cS'\subset \cS,|\cS'|\le d:\\ \proj([d]) \subset \proj(\cS')}} \sum_{S\in \cS'}(|S|-1-\frac{d-3}{d})=d\,.
\]
Use induction on $d$. For $d=3$ there is only one possible $\cS'$, it's easy to verify that $h(3)=3$. 

Now assume $h(d-1) = d-1$, we want to show $h(d)=d$. For the simplicity of discussion, let $D(d)$ be the set of $\cS'$ that satisfy $|\cS'|\le d$ and $\proj([d]) \subset \proj(\cS')$. Also define a functional $F(\cS') \defeq \sum_{S\in \cS'}(|S|-1-\frac{d-3}{d})$. For any $\cS'\in D(d)$ and any $v\in[d]$, let us define a mapping $M_v(\cS'):D(d)\rightarrow D(d-1)\cup \{\bot\}$,
\[
M_v(\cS') = 
\begin{cases}
    \bot &\text{ if } ([d]\backslash\{v\})\in \cS'\\
    \{S\backslash\{v\}:S\in \cS', |S\backslash\{v\}|>1\} &\text{ otherwise } 
\end{cases}
\]
For now assume there exists a $v\in [d]$ that satisfy the two following properties,
\begin{itemize}
    \item $([d]\backslash\{v\})\not\in \cS'$ and
    \item $|\{S\in \cS': |S|=2,v\in S\}|\le 1$.
\end{itemize}
We will prove why such $v$ exists later. We have
\[
\begin{split}
&F(\cS')-F(M_v(\cS'))\\
&=\sum_{S\in \cS'}(|S|-1-\frac{d-3}{d}) - \sum_{S\in M_v(\cS')}(|S|-1-\frac{d-4}{d-1})\\
&=(\frac{d-4}{d-1}-\frac{d-3}{d})|\cS'|+\sum_{S\in \cS'}(|S|-1-\frac{d-4}{d-1})- \sum_{S\in M_v(\cS')}(|S|-1-\frac{d-4}{d-1})\\
&\ge -\frac{3}{d-1} +\sum_{S\in \cS',v\in S,|S|>2}1 +\sum_{S\in \cS',v\in S,|S|=2}(1-\frac{d-4}{d-1})\\
&= -\frac{3}{d-1} +\sum_{S\in \cS',v\in S}1 - \sum_{S\in \cS',v\in S,|S|=2}\frac{d-4}{d-1}\,.
\end{split}
\]
For the inequality, we used that $|\cS'|\le d$. By the assumption on $v$, we have $|S\in \cS',v\in S,|S|=2|\le 1$. Also $ v$ cannot only be contained in one set $S$ in $\cS'$, otherwise as for any vertex $u\ne v$, pair $\{v,u\}$ is contained in some set, $S$ would contain all nodes in $[d]$, contradicting the fact that $\cS$ only contain proper subsets of $[d]$. So $|S\in \cS',v\in S|\ge 2$. Taking all these into the inequality above, we have
\[
F(\cS')-F(M_v(\cS'))\ge  -\frac{3}{d-1}+2-\frac{d-4}{d-1} = 1\,.
\]
By induction assumption, $F(M_v(\cS'))\ge d-1$, so $F(\cS')\ge d$. This holds for any $\cS'\in D(d)$, so $h(d)\ge d$. As $d$ can be achieved by the configuration in the lemma statement, we have $h(d)=d$.

What remains is to show there exists $v$ that satisfy $([d]\backslash\{v\})\not\in \cS'$ and
 $|\{S\in \cS': |S|=2,v\in S\}|\le 1$. For contradiction suppose every $v$ either satisfy $([d]\backslash\{v\})\in \cS'$ or $|\{S\in \cS': |S|=2,v\in S\}|\ge 2$. Suppose there are $s$ nodes in $[d]$ satisfy $|\{S\in \cS': |S|=2,v\in S\}|\ge 2$, denote this set of nodes by $T$, and at least $d-s$ nodes satisfy $([d]\backslash\{v\})\in \cS'$, denote this set of nodes $U$. Then there are at least $d-s$ sets in $\cS'$ with size $d-1$. By counting degree, there are at least $2s/2=s$ size-2 set in $\cS'$. Since $|\cS'|=d$, there must be $d-s$ sets  with size $d-1$ and $s$ sets with size 2  in $\cS'$. The $s$ size-2 sets must be between nodes in $T$. So nodes in $U$ are not in any size-2 sets. If $d-s=1$, then that node in $U$ is not connected to any node, contradiction. If $d-s=2$, the two nodes in $U$ are not connected. Therefore, $d-s\ge 3$, there are at least 3 size $d-1$ hyperedges. In this case $F(\cS')>d$ and can be omitted as we only care about the minimum value of $F(\cS')$. 
 \end{proof}


%% file: deferred.tex
\section{Deferred Proofs of Lemmas}
\label{s:deferred}

\subsection{Proof of Lemma~\ref{lem:monotone}}\label{sec:monotone}
Let $p_1 = n^{-d+1+\delta_1}$ and $p_2 = n^{-d+1+\delta_2}$. Let $\rhG_1$ and $\rhG_2$ be the random hypergraphs when hyperedge density are $p_1$ and $p_2$ respectively. Assume we have a black-box algorithm that exactly recovers $\rhG_2$. We will use it  to  recover $\rhG_1$ from $\proj(\hE_{\rhG_1})$.

The key observation is that a dense graph is the union of two sparse graphs. Specifically, let $p_3$ satisfy $p_1+(1-p_1)p_3 = p_2$, and $\rhG_3$ be a random hypergraph sampled from $\rhG(n,d,p_3)$. In the union  $\hE_{\rhG_1}\cup \hE_{\rhG_3}$, each hyperedge is included with probability $p_1+(1-p_1)p_3=p_2$. We have $\hE_{\rhG_1}\cup \hE_{\rhG_3}$ and $\hE_{\rhG_2}$ follows the same distribution.

Now given $\proj(\hE_{\rhG_1})$, we generate a sample of $\rhG_3$. Using the black box, we can recover $\hE_{\rhG_1}\cup \hE_{\rhG_3}$ from
\[
\proj (\hE_{\rhG_1}\cup \hE_{\rhG_3}) = \proj(\hE_{\rhG_1}) \cup \proj(\hE_{\rhG_3})\,
\]
with high probability.
By union bound over all possible hyperedges, probability that $\hE_{\rhG_1}$ and $\hE_{\rhG_3}$ has non-empty overlap is upper bounded by
\[
\Pr(\hE_{\rhG_1}\cap \hE_{\rhG_3}\ne \emptyset)\le 
\binom{n}{d}p_1p_3\le n^d p_1p_2\le n^{-d+2+\delta_1+\delta_2}=o_n(1)\,.
\]
Here we used $p_1<p_3<p_2$ in the first inequality. The last equality follows from $d\ge 4$ and $\delta_1<1$, $\delta_2\le 1$. So with high probability, $\hE_{\rhG_1}$ and $\hE_{\rhG_3}$ do not have common hyperedges, and we can recover $\hE_{\rhG_1}$ by subtracting $\hE_{\rhG_3}$ from $\hE_{\rhG_1}\cup \hE_{\rhG_3}$.
\hfill\qed

\subsection{Proof of Lemma~\ref{lem:union-min-preimage}}\label{sec:union-min-preimage}
Let  $E_i = \proj(C_i)$, and $V_i$ be the node set of $C_i$. Recall that for any two sets of hyperedges, $\proj(C_1\cup C_2) = \proj(C_1)\cup \proj( C_2)$.

\paragraph{Step 1: $\cup_i \proj(C_i)$ forms a partition of $ E_p$.} 
First, note that because $(C_i)_i$ partitions the hyperedges in $\cliG$, and then by definition of $\cliG$, 
$$
\bigcup_i \proj(C_i) = \proj(\cliG) \,.
$$
Next, if any $\{a,b\}\in C_i\cap C_j$, then there are hyperedges $h_i\in C_i$ and $h_j\in C_j$ each containing $\{a,b\}$. But then $h_i$ and $h_j$ are 2-connected, so $C_i$ and $C_j$ cannot be two separate components. 

\paragraph{Step 2: LHS $\subseteq$ RHS.} 
 Consider an arbitrary preimage $\hG\in \prim (\pE)$.
 We have $\hG\subseteq \cliG$, because $\hG$ is a clique cover of $E_p$ and $\cliG$ is the maximal clique cover. Since $\{C_i\}_i$ partitions the hyperedges in $\cliG$, 
$$H
=  \bigcup_{i\in [m]}\hG\cap C_i\,.$$
We will now argue that $\hG\cap C_i\in \prim(\proj(C_i))$, i.e., $\proj(\hG\cap C_i)= \proj(C_i)$. 
It suffices to prove that $\proj(\hG\cap C_i)\supseteq \proj(C_i)$. Consider any edge $e=\{a,b\}\in \proj(C_i)\subseteq E_p$. 
Now,
$\hG$ has some hyperedge $h$ containing the endpoints of $e$, because $\proj(\hG) = E_p$. 
Secondly, 
$C_i$ has all hyperedges in $\cliG$ containing $e$,
by step 1,
and therefore contains $h$. It follows that $\proj(\hG\cap C_i)$ contains $e$.


\paragraph{Step 3: RHS $\subseteq$ LHS.} 
 We want to show that any union on the right-hand side is in $\prim(E_p)$.
This follows immediately from step 1: if $H = \cup_{i=1}^m H_i$ for $\hG_i\in \prim(\proj(C_i))$, then $\proj(H) = \cup_i \proj(H_i) = \cup_i \proj(C_i) = E_p$. \qed

\subsection{Proof of Lemma~\ref{lem:subgraph_threshold}}\label{sec:subgraph_threshold}
Let us first state a lemma that will be used later.
\begin{lemma}\label{lem:non-zero}
    For a real-valued random variable $X$,
    \[
    \Pr(X\ne 0) \ge \frac{(\E X)^2}{(\E X)^2+\var(X)}\,.
    \]
\end{lemma}
\begin{proof}
Let $Y=\ind{X\ne 0}$. By Cauchy-Schwartz, 
\[
(\E XY)^2\le \E [X^2]\E[Y^2]\,.
\]
Since $XY=X$, the left hand side is equal to  $(\E X)^2$. Since $Y$ takes 0,1 value, $\E[Y^2]=\E[Y]=\Pr(X\ne 0)$. We have
\[
\Pr(X\ne 0)\ge \frac{(\E X)^2}{\E [X^2]} = \frac{(\E X)^2}{(\E X)^2+\var(X)}\,.
\]
\end{proof}

We prove the second and third claim using second moment method

Let $\shG_1,\shG_2,\cdots,\shG_t$ be all copies of such sub-hypergraph on the complete graph of $[n]$, we have
\[
t=\binom{n}{v_\shG}\frac{(v_\shG)!}{\aut(\shG)}=\Theta_n(n^{v_\shG})\,.
\]
Here $\aut(\shG)$ is the number of automorphisms of $\shG$.
Let $I_i$ be the indicator that $\shG_i$ is in $\rhG$. And $X_\shG = \sum_{i=1}^tI_i$ be the number of such event happening. We have
\[
\E [X_\shG] = tp^{e_\shG}(1-p) = \Theta_n(n^{v_\shG}p^{e_\shG})\,.
\]
And 
\[
\var(X_\shG) = \sum_{i=1}^t\sum_{j=1}^t \cov(I_iI_j) = 
\sum_{i=1}^t\sum_{j=1}^t (\Pr(I_i=I_j=1)-\Pr(I_i=1)\Pr(I_j=1))\,.
\]
We have $\Pr(I_i=1) = \Pr(I_j=1) = p^{d}(1-p)$. 
Consider pairs $(\shG_i,\shG_j)$ such that $\shG_i\cap \shG_j = \shG'$, where $\shG'\subset \shG$ is a sub-hypergraph of $\shG$ with non-empty edge set.
\begin{align*}
\var(X_\shG) &= O_n\B( \sum_{\substack{\shG' \subseteq \shG,\\ e_{\shG'} > 0}} n^{2v_\shG - v_{\shG'}} \left( p^{2e_\shG - e_{\shG'}} - p^{2e_\shG} \right) \B) \\
&= O_n\B( n^{2v_\shG} p^{2e_\shG} \sum_{\substack{\shG' \leq \shG,\\ e_{\shG'} > 0}} n^{-v_{\shG'}} p^{-e_{\shG'}} \B).
\end{align*}
Then from Lemma~\ref{lem:non-zero}, 
\[
\Pr(X_\shG\ne 0) \ge \frac{(\E X_\shG)^2}{(\E X_\shG)^2+\var(X_\shG)} = \frac{1}{1+O_n(\sum_{\substack{\shG' \leq \shG,\\ e_{\shG'} > 0}} n^{-v_{\shG'}} p^{-e_{\shG'}})}\,.
\]
We can easily check that for any $\shG'\subset \shG$, $\frac{e_\shG'}{v_\shG'}\le \frac{e_\shG}{v_\shG}$. So when $p=\Theta_n(n^{-1/m(\shG)}) = \Omega_n (n^{-v_{\shG'}/e_{\shG'}})$, we have
\[
\sum_{\substack{\shG' \subseteq \shG,\\ e_{\shG'} > 0}} n^{-v_\shG'} p^{-e_{\shG'}} = \sum_{\substack{\shG' \leq \shG,\\ e_\shG' > 0}} O_n(n^{-v_{\shG'}} n^{v_{\shG'} }) = O_n(1)\,.
\]
This means $\Pr(X_\shG\ne 0) = \Omega_n(1)$.

When $p=\omega_n(n^{-1/m(\shG)}) = \omega (n^{-v_\shG/e_\shG})$, 
we have
\[
\sum_{\substack{\shG' \leq \shG,\\ e_\shG' > 0}} n^{-v_\shG'} p^{-e_\shG'} = \sum_{\substack{\shG' \leq \shG,\\ e_\shG' > 0}} o_n(n^{-v_{\shG'}} n^{v_{\shG'} }) = o_n(1)\,.
\]
This means $\Pr(X_\shG\ne 0) = 1-o_n(1)$.

Next we prove the first claim. Let $\shG'\subset \shG$ be the sub-hypergraph that $\frac{e_{\shG'}}{v_{\shG'}} = m(\shG)$.
\[
\Pr(\shG\subset \rhG)\le \Pr(\shG'\subset \rhG)\le \E X_{\shG'} = \Theta_n(n^{v_{\shG'}}p^{e_{\shG'}})\,.
\]
When $p=o_n(n^{-1/m(\shG)})$ the above is $o_n(1)$.
\hfill\qed

\subsection{Proof of Lemma~\ref{lem:no-ambiguous}}\label{sec:proof-no-ambiguous}
Let $C_1,\cdots, C_m$ be all the 2-connected components in $\cliG=\cli(\pG)$. Let $V_i$ be the node set of $C_i$, $r_i$ be the size of the minimum preimage of $C_i$. The success probability can be written in terms of the posterior distribution.
\[
\Pr(\cA^*(\pG)=\rhG) = \E_{\pG}[p_{\rhG|\pG}(\cA^*(\pG)|\pG)] \,.
\]
We will show that this posterior probability is close to 1 with high probability.
Recall the posterior distribution is
\[
p_{\rhG|\pG}(H|\pG) =\frac{\ind{\proj(\hE_H) = \pE}p_{\rhG}(\hE_H) }{p_\pG(\pE)}\propto  \ind{\hE_H \in \prim( \pE)} \b(\frac{p}{1-p}\b)^{|\hE_H|}\,,
\]
By Lemma~\ref{lem:union-min-preimage}, $\hE_H \in \prim( \pE)$ is equivalent to $H\cap C_i\in \prim(\proj(C_i))$ for all $i$. Recall that $\hE_H=\cup_i(H\cap C_i)$. So the posterior distribution can be written as
\[
p_{\rhG|\pG}(H|\pG) \propto  \prod_{i=1}^m \B(\ind{H\cap C_i\in \prim(\proj(C_i))} \b(\frac{p}{1-p}\b)^{e(H\cap C_i)}\B)\,.
\]
Here $e(H\cap C_i)$ stands for the number of hyperedges in $H\cap C_i$. We have
\[
p_{\rhG|\pG}(\cA^*(\pG)|\pG) = \prod_{i=1}^m \frac{ \b(\frac{p}{1-p}\b)^{r_i}}{\sum_{H'\in \prim(\proj(C_i))} \b(\frac{p}{1-p}\b)^{e(H')}} =  \prod_{i=1}^m \frac{ 1}{\sum_{H'\in \prim(\proj(C_i))} \b(\frac{p}{1-p}\b)^{e(H')-r_i}}\,.
\]

By Lemma~\ref{lem:component-constant-size}, with high probability any $C_i$ has size at most $(2^d+1)/(\frac{d-1}{d+1}-\delta)$. So $|V_i|\le d(2^d+1)/(\frac{d-1}{d+1}-\delta)$. By the assumption of the lemma, any ambiguous graph $G_a$ with at most $(d2^d+1)/(\frac{d-1}{d+1}-\delta)=O_n(1)$ number of nodes has $o_n(1)$ probability of appearing in $\pG$. So by union bound, the probability of any such $G_a$ appearing in $\pG$ is $o_n(1)$.
Therefore, with probability $1-o_n(1)$, $\proj(C_i)$ is not ambiguous for any $i$. This means there is only one hypergraph in $\prim(\proj(C_i))$ with size $r_i$. So with probability $1-o_n(1)$, 
\[
\sum_{H'\in \prim(\proj(C_i))} \b(\frac{p}{1-p}\b)^{e(H')-r_i} \le 1+|\prim(\proj(C_i))|\frac{p}{1-p}=1+O_n(p)\,.
\]
The last equality is because $C_i$ is of size $O_n(1)$, so the number of possible preimages is also $O_n(1)$. Taking this back to the expression of posterior probability, we get
\[
p_{\rhG|\pG}(\cA^*(\pG)|\pG) = (1-O_n(p))^m=1-O_n(mp)\,.
\]
$m$ is the number of 2-connected component, which is bounded by the total number of hyperedges in $\rhG$. On the other hand, the total number of hyperedges in $\rhG$ follows binomial distribution $\bino(\binom{n}{d},p)$. By Chernoff bound, it is $\Theta(n^dp)$ with probability $1-o_n(1)$. So we have 
\[
p_{\rhG|\pG}(\cA^*(\pG)|\pG) = 1-o_n(1)-O_n(n^dp^2)=1-o_n(1)-O_n(n^{-d+2+2\delta})\,.
\]
Since $d\ge 3$, $\delta < \frac{d-1}{d+1}\le \frac{1}{2}$, we have $-d+2+2\delta<0$. So with high probability 
\[
p_{\rhG|\pG}(\cA^*(\pG)|\pG) = 1-o_n(1)\,,
\]
and therefore
\[
\Pr(\cA^*(\pG)=\rhG) = \E_{\pG}[p_{\rhG|\pG}(\cA^*(\pG)|\pG)] =1-o_n(1)\,.
\]
\hfill\qed

\subsection{Proof of Lemma~\ref{lem:density_bad_graph}}\label{sec:proof_density_bad_graph}
Recall the definition of $m(S_1\cup S_2\cup \{\he_1\})$ is 
\[
\max_{K\subset (S_1\cup S_2\cup \{\he_1\})} \frac{e_K}{v_K}
\]
Below we show that this is reached by the whole hypergraph, i.e., when $K=S_1\cup S_2\cup \{\he_1\}$. 
Let $L$ be the set of hyperedges in $S_1$ that is a subset of $K$, $R$ be the set of hyperedges in $S_2$ that is a subset of $K$.

Case 1: $\he_1\not\in K$, $R=\emptyset$. 
\[
\frac{e_K}{v_K} = \frac{|L|}{(d-1)|L|+1}\le \frac{d-1}{(d-1)^2+1}\,.
\]
The maximum is achieved when $L=S_1$.
The case where $R\ne \emptyset$ and $L=\emptyset$ is symmetric. 

Case 2: $\he_1\not\in K$, $L,R\ne \emptyset$. Without loss of generality, assume $|L|\ge |R|$.
\[
\begin{split}
    \frac{e_K}{v_K}&= \frac{|L|+|R|}{(d-1)(|L|+|R|)+2 - \#[i:\he_i^w\in L,\he_i^z\in R]}\\
    &\le \frac{2|L|}{2(d-1)|L|+2-|L|}\\
    &\le \frac{2d-2}{(d-1)(2d-2)+2-(d-1)}\,.
\end{split}
\]
The maximum is achieved when $L=S_1$ and $R=S_2$. Easy to see that maximum in case 2 is larger than the maximum in case 1. 

Case 3: $\he_1\in K$, $R=\emptyset$.
\[
\frac{e_K}{v_K} = \frac{1+|L|}{d+|L|(d-2)}\le \frac{d}{d+(d-1)(d-2)}\,.
\]
The maximum is achieved when $L=S_1$.
The case where $R\ne \emptyset$ and $L=\emptyset$ is symmetric. 

Case 4: $\he_1\in K$, $L,R\ne \emptyset$. Without loss of generality, assume $|L|\ge |R|$.
\[
\begin{split}
    \frac{e_K}{v_K}&= \frac{1+|L|+|R|}{d+1+|L|(d-2)+|R|(d-2)}\\
    &\le \frac{2d-1}{d+1+(2d-2)(d-2)} = \frac{2d-1}{2d^2-5d+5}\,.
\end{split}
\]
The maximum is achieved when $L=S_1$ and $R=S_2$. It is easy to see that the maximum in case 4 is larger than the maximum in case 3. The maximum in case 4 has one more hyperedge than the maximum in case 2, which does not increase the number of nodes. Therefore, case 4 is the maximum overall and $m(S_1\cup S_2\cup \{\he_1\}) = \frac{2d-1}{2d^2-5d+5}$. \hfill\qed

\subsection{Proof of Lemma~\ref{lem:branching}}\label{sec:brahcing}

The high-level approach is to union bound over all possible hyperedges in $\nei{\cliG}(\cli(\hE_1))$. Let $V(\hE_1)$ be the set of nodes that are incident to one of the hyperedges in $\hE_1$. Further, let $A_k$ be the set of hyperedges that has $k$ nodes in $V(\hE_1)$, i.e., 
\[
A_k \defeq \B\{\he\in \binom{[n]}{d}\b| |\he\cap V(\hE_1)|=k\B\}\,.
\]
Here $k$ is at least 2 and at most $d$. The size of $A_k$ is at most
\[
|A_k|\le \binom{|V(\hE_1)|}{k}\binom{n}{d-k}=O_n(n^{d-k})\,.
\]
We wish to union bound the probability that any hyperedge $\he\in A_k$ being present in $\cliG.$

Let $\he\in A_k$. For $\he$ to appear in $\cliG$, every edge in $\proj(\he)\backslash \proj(\hE_1)$ should be covered in at least one hyperedge in $\hE_{\rhG}$.
Now let us look at the possible ways for this to happen.  
For any $\he\in A_k$, let $\cS^\he \defeq \{S_1^\he, S_2^\he, \cdots, S_m^\he\}$ be the set of subset of $\he$ such that $\proj(S_i)\not\subset \proj(\hE_1)$ and $\proj(S_i)\ne \emptyset$. 
If a hyperedge covers an edge in $\proj(\he)\backslash \proj(\hE_1)$, it must intersect with $h$ at one of the sets in $\cS^\he$.
Now let event $A_i^\he$ be the event that at least one hyperedge in 
\[
\hE_i^\he\defeq \B\{\he'\in \binom{[n]}{d}\b| \he'\cap \he = S_i^\he\B\}\,
\]
is in $\hE_\rhG$. Note that $\{S_i^\he\}_i$ are disjoint set of hyperedges, so $\{A_i^\he\}_i$ are independent events.
We have
\begin{equation}\label{eq:decom-inclusion}
\begin{split}
&\Pr(\he\in \cliG\b|\hE_1\in\hE_\rhG)\\ 
&=  \sum_{I\subset [m ]} \1\{(\proj(\he)\backslash \proj(\hE_1)) \subset \cup_i\proj(S_i^\he)\} \Pr\b((\cap_{i\in I}A_i^\he)\cap(\cap_{i\in [m ]\backslash [I]}(A_i^\he)^c) |\hE_1\subset \hE_\rhG\b)\\
&\le \sum_{I\subset [m ]} \1\{(\proj(\he)\backslash \proj(\hE_1)) \subset \cup_i\proj(S_i^\he)\} \Pr(\cap_{i\in I}A_i^\he|\hE_1\subset \hE_\rhG)\\
&=  \sum_{I\subset [m ]} \1\{(\proj(\he)\backslash \proj(\hE_1)) \subset \cup_i\proj(S_i^\he)\} \prod_{i\in I}\Pr(A_i^\he|\hE_1\subset \hE_\rhG)\,.
\end{split}
\end{equation}
The inequality is by inclusion of events, and the second equality is by independence of $\{A_i^\he\}_i$.
Now we show an upperbound on $\Pr(A_i^\he|\hE_1\subset \hE_\rhG)$. There are $\binom{n-|S_i^\he|}{d-|S_i^\he|}$ hyperedges in $\hE_i^\he$. And none of them are in $\hE_1$ since $\proj(S_i)\not\subset \proj(\hE_1)$. Therefore, 
\[
\Pr(A_i^\he|\hE_1\subset \hE_\rhG) = 1-(1-p)^{\binom{n-|S_i^\he|}{d-|S_i^\he|}} = O_n(pn^{d-|S_i^\he|}) = O_n(n^{-|S_i^\he|+1+\delta})\,.
\]
Note that $|S_i^\he|\ge 2$ and $\delta< 1$, this is always $o_n(1)$.
Since the number of terms in \eqref{eq:decom-inclusion} is bounded by $2^m\le 2^{2^d}$ which is $O_n(1)$,  we have
\begin{equation}\label{eq:braching-hyperedge}
\Pr(\he\in \cliG\b|\hE_1\in\hE_\rhG) = O_n\B(\max_{ \substack{I\subset[m]:\\(\proj(\he)\backslash \proj(\hE_1)) \subset \cup_i\proj(S_i^\he)}} n^{-\sum_{i\in I}(|S_i^\he|-1-\delta)}\B)\,.
\end{equation}
Because $\he\in A_k$,  we know $\proj(\he)\cap \proj(\hE_1)$ is a subset of a size-$k$ clique in $\he$. So the above probability can be further relaxed to $O_n(n^{-g_k(\delta)})$. Recall
\begin{equation}\label{eq:braching-hyperedge-2}
g_k(\delta) = \min_{\substack{I\subset [m]:\\ \b(\proj(\he)\backslash \binom{U_\he}{2}\b) \subset \cup_i\proj(S_i^\he)}} \sum_{i\in I}(|S_i^\he|-1-\delta)\,.
\end{equation}
Here $U_\he$ is a size-$k$ subset of $\he$. Note that any clique in $\cS^\he$ has size at least 2, $g_k(\delta)$ is always non-negative.
Therefore, by union bound over all hyperedges in $A_k$ for any $2\le k\le d$, 
\[
\Pr\b(\nei{\cliG}(\cli(\hE_1))\ne \emptyset | \hE_1\subset \hE_{\rhG}\b) = \sum_{k=2}^d |A_k|O_n(n^{-g_k(\delta)}) = O_n(n^{\min_k \{g_k(\delta)+k-d\}})\,.
\]

Given the bound for $\min_k \{g_k(\delta)+k-d\}$ in Lemma~\ref{lem:cover-bound}, we have for any $\delta<\frac{d-1}{d+1}$,
\[
\Pr\b(\nei{\cliG}(\cli(\hE_1))\ne \emptyset | \hE_1\subset \hE_{\rhG}\b) = O_n(n^{-\frac{d-1}{d+1}+ \delta})\,.
\]

Now we prove the case when $\delta = \frac{d-1}{d+1}$. In this case, instead of using union bound, we need to be more careful and consider the correlation between different hyperedges using Harris Inequality. 

\begin{lemma}[Harris Inequality \cite{harris1960lower}]\label{lem:harris}
Let $A$ and $B$ be two events in the probability space defined by $\rhG$. If both $A$ and $B$ are increasing with respect to all possible hyperedges in $\binom{[n]}{d}$, then
\[
\Pr(A|B)\ge \Pr(A)\,.
\]
\end{lemma}

Let $A=\{\he_1,\he_2,\cdots,\he_m\}$ be the set of all hyperedges in $\nei{\cliG}(\cli(\hE_1))$, $A = \cup_{k=1}^{d|\hE_1|}A_k$. We have
\[
\begin{split}
    &\Pr\b(\nei{\cliG}(\cli(\hE_1))= \emptyset | \hE_1\subset \hE_{\rhG}\b)\\
    &=\prod_{i=1}^m \Pr\b( \he_i\not\in \hE_\rhG| \forall j<i,\he_j\not\in \hE_\rhG, \hE_1\subset \hE_{\rhG}\b)\\
    &\ge \prod_{i=1}^m \Pr\b( \he_i\not\in \hE_\rhG|  \hE_1\subset \hE_{\rhG}\b) \\
    &= \prod_{k=2}^{d}\prod_{h\in A_k}\Pr\b( \he\not\in \hE_\rhG|  \hE_1\subset \hE_{\rhG}\b)
\end{split}
\]
Here we used that $\he_i\not\in \hE_\rhG$ is decreasing event for any $i$ and applied Harris Inequality in Lemma~\ref{lem:harris}. Using the bound we get in \eqref{eq:braching-hyperedge} and \eqref{eq:braching-hyperedge-2}, we have
\[
\Pr\b(\nei{\cliG}(\cli(\hE_1))= \emptyset | \hE_1\subset \hE_{\rhG}\b) \ge \prod_{k=2}^{d} \b(1-O_n(n^{-g_k(\delta)})\b)^{|A_k|} = \prod_{k=2}^{d} \exp \B(-O_n(n^{d-k-g_k(\delta)})\B)\,.
\]
By Lemma~\ref{lem:cover-bound}, when $\delta=\frac{d-1}{d+1}$, $d-k-g_k(\delta)\ge 0$. So $\Pr\b(\nei{\cliG}(\cli(\hE_1))= \emptyset | \hE_1\subset \hE_{\rhG}\b)=\Omega_n(1)$, thus proving the second case stated in Lemma~\ref{lem:branching}. \hfill\qed